\documentclass[final,leqno,onefignum,onetabnum]{siamltex1213}

\usepackage{amsmath}
\usepackage{amsfonts}
\usepackage{amssymb,xcolor}
\usepackage{textcomp}
\usepackage{wrapfig}
\usepackage{graphicx}
\usepackage{enumerate}
\usepackage{cases}
\usepackage{verbatim}
\usepackage{bm}
\usepackage{multicol}
\usepackage{hyperref}
\usepackage{lineno}
\usepackage{textcomp}
\usepackage[norelsize,english,ruled,lined,vlined,commentsnumbered]{algorithm2e}
\usepackage{tikz,pgf}
\usetikzlibrary{shapes,snakes,arrows,patterns,trees,datavisualization}

\DeclareMathOperator{\curl}{curl}
\renewcommand{\grad}{\nabla}
\newcommand{\divv}{\mathrm{div} \,}

\newcommand{\R}{\mathbb{R}}
\newcommand{\C}{\mathbb{C}}

\newcommand{\bmH}{\textbf{H}}
\newcommand{\bmE}{\textbf{E}}
\newcommand{\bmJ}{\textbf{J}}
\newcommand{\bmA}{\textbf{A}}
\newcommand{\bmPsi}{\bm{\Psi}}
\newcommand{\bmvarPsi}{\bm{\varPsi}}
\newcommand{\bmB}{\textbf{B}}

\newcommand{\bmeta}{\bm{\eta}}
\newcommand{\bmn}{\bm{n}}
\newcommand{\bmtheta}{\bm{\theta}}
\newcommand{\bmP}{\textbf{P}}

\newcommand{\sfa}{\mathsf{A}}
\newcommand{\sfaprime}{\mathsf{A}'}
\newcommand{\sfaI}{\mathsf{A}_{i}}
\newcommand{\sfaprimeI}{\mathsf{A}'_{i}}
\newcommand{\sfaZero}{\mathsf{A}_{0}}
\newcommand{\sfaprimeZero}{\mathsf{A}'_{0}}

\renewcommand*{\d}[1]{\thinspace \dif#1}
\newcommand{\dif}{\mathrm{d}}

\renewcommand*{\supp}{\mathrm{supp}}

\newcommand{\mi}{\mathrm{i}}
\newcommand{\Id}{\mathrm{Id}}

\newcommand{\X}{\mathcal{Q}} 

\title{A robust inversion method for quantitative 3D shape reconstruction from
  coaxial eddy-current measurements }

\author{Houssem Haddar\footnotemark[1]
\and Zixian Jiang\footnotemark[2]
\and Mohamed Kamel RIAHI\footnotemark[3]}

\begin{document}
\maketitle
\slugger{sisc}{xxxx}{xx}{x}{x--x}

\renewcommand{\thefootnote}{\fnsymbol{footnote}}
\footnotetext[1]{CMAP, Ecole Polytechnique, Route de Saclay, 91128 Palaiseau Cedex FRANCE.
(\email{haddar@cmap.polytchnique.fr})}
\footnotetext[2]{Centre for Industrial Mathmatics, Universitaet Bremen, Germany.
(\email{jiang@math.uni-bremen.de})}
\footnotetext[3]{Department of mathematical science, New Jersey Institute of Technology, University Heights Newark, New Jersey, USA.(\email{riahi@njit.edu})}

\begin{abstract}
	This work is motivated by the monitoring of conductive clogging deposits
        in steam generator at the level of support plates. One would like to
        use monoaxial coils measurements to obtain estimates on the clogging
        volume. We propose a 3D shape optimization technique based on simplified
        parametrization of the geometry adapted to the measurement
        nature and resolution. The direct problem is modeled by the eddy current approximation
        of time-harmonic Maxwell's equations in the low frequency regime. A
        potential formulation is adopted in order to easily handle the complex
        topology of the industrial problem setting. We first characterize the
         shape derivatives of the deposit impedance signal using an adjoint field technique. For the
        inversion procedure, the direct and adjoint
        problems have to be solved for each coil vertical position which is
        excessively time and memory consuming.  To overcome this difficulty, we propose and discuss a
        steepest descent method based on a fixed and invariant
        triangulation. Numerical experiments are presented to illustrate the convergence and the efficiency of the method.
\end{abstract}



\pagestyle{myheadings}
\thispagestyle{plain}
\markboth{3D shape reconstruction from eddy-current measurements}{H. Haddar, Zixian Jiang and M.-K RIAHI}

\section{Introduction}\label{section:introduction}
Non-destructive testing using eddy-current low frequency excitation are widely practiced to detect magnetite deposits in steam generators (SG) in nuclear power plants. These deposits, due to magnetite particles contained in the cooling water, usually accumulate around the quatrefoil support plates (SP) and thus clog the water traffic lane. Many methods and softwares based on signal processing has been developed in order to detect deposits using standard bobbin coils and are widely operational in the nuclear industries (see for instance the database of nondestructive testing \cite{ndtdatabase} and references therein). 
Estimates of the bulk amount of deposits enable to supplement a chemical cleaning process, which in some cases may be ineffective where it leaves significant deposits in the bottom area of the SP foils. The presence of such deposits generates a reduction and re-distribution of the water in SG circulation and can cause flow-induced vibration instability risks. This may harm the safety of the nuclear power plant.

\begin{figure}[!hbp]
\centering
\begin{tabular}{lr}
\includegraphics[scale=.2] {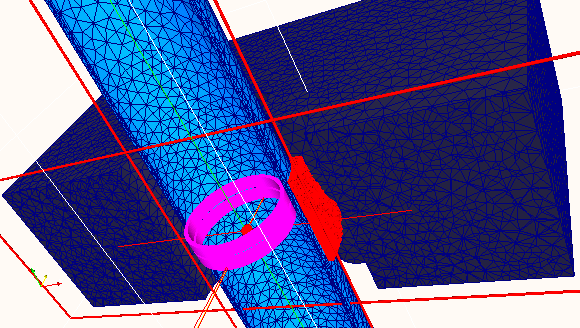}&\hfill
\includegraphics[scale=.2] {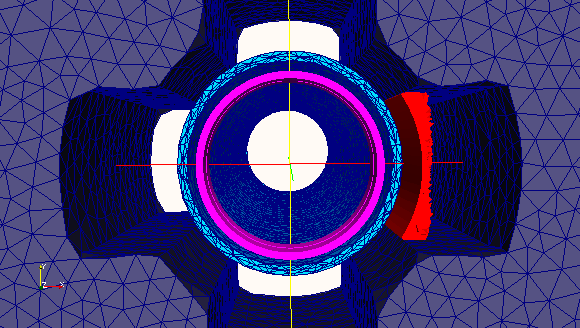}
\end{tabular}
\caption{Three dimensional mesh of the SG and the SP clogging : coils (pink), tube (blue), SP (grey) and deposit (red).}
\label{figintro}
\end{figure}

	In order to obtain better characterizations than those provided  by
        model free methods, we present and discuss a robust inversion algorithm
        (for
        non destructive evaluation using eddy current signals) based on shape
        optimization techniques and adapted parametrizations for the deposit
        shapes.  An overview of techniques for non destructive evaluations using eddy currents  can be found in~\cite{auld1999review} and we also
        refer to~\cite{theddcurrinvers},~\cite{1211580} and \cite{1386227} for
        further engineering considerations. For other model based inversion methods related to
        eddy-currents  we may refer, without being exhaustive,
        to~\cite{bendjoudi:hal-00549199,972391,877774,RPQNE, MR3086051}. In the medical context, several inverse source problems related to
        eddy-current models have been addressed: non-invasive applications for
        electroencephalography, magnetoencephalography\cite{0266-5611-28-1-015006} (see also~\cite{MR2235652}) and magnetic induction tomography~\cite{griffiths2001magnetic,MLHGHB}. 

Stated more precisely, the inverse shape problem we shall investigate aims at retrieving the
support of a conductive deposits using monostatic measurements of coaxial coils
and their computable shape derivatives. Our work can be
seen as an extension of~\cite{jiang:hal-00741616} to a realistic 3D industrial
configuration. Although the deposit geometry can be an arbitrary three
dimensional domain, the
available (monostatic) measurements can only give qualitative information on
the width. Since the objective is to detect the possibility of clogging at the
support plate, we found it appropriate to consider  a deposit concentrated in
only one of the opening regions at the support plate (See Figs.~\ref{figintro}). The geometrical parameters are then the
deposit width at (at most) one measurement position. In practice, it turned out
that a relative robustness with respect to noise can be achieved if one shape
parameter correspond with two vertical positions of the coils. In order to
speed up
the inversion procedure, we are led to consider a fixed geometrical mesh
(adapted to the chosen parametrization). This allows us to obtain an inversion
procedure which is not very sensitive to the number of measurements. Moreover,
in order to avoid troubles due to changes in the conductive region topology, we
adopted a vector potential formulation of the 3D eddy-current model. A careful
study of the shape derivative of the solution to this formulation is
conducted. For related shape derivatives associated with Maxwell's equations we
refer to \cite{MR2738932,MR2971616}. We here treat the potential formulation of the eddy current problem. This derivative allows us to rigorously define the adjoint state,
needed to cheaply compute the coils impedances shape derivatives.

The geometrical setting of the industrial configuration is depicted in
Fig.~\ref{figintro}. We denote by $\Omega$ the computational domain, which will be a
sufficiently large simply connected cylinder. It contains  a conductor domain
$\Omega_{\mathcal{C}}$ composed of the tube, the support plate and eventually a
deposit on the exterior part of the tube: $\Omega_{\mathcal{C}} = \Omega_{t} \cup \Omega_{d} \cup
        \Omega_{p}$, where $t$ stands for the tube, $d$ for the deposit and $p$
        for the SP. The insulator domain $\Omega\setminus\overline\Omega_{\mathcal{C}}$ is split into two parts: $ \Omega_{s} $ that indicates the region inside the tube  where the coil (thus the source $\bmJ$) is located and $\Omega_{v}$ that  denotes the insulator outer region (where the deposit can be formed).
        For our purpose we introduce the surface $\Gamma =
        \overline{\partial\Omega_{d}} \cap \overline{\partial\Omega_{v}}$ that denotes the interface between the deposit and the insulator.
	
	

	Let us now briefly describe the 3D eddy-current model, which derives
        from the full Maxwell's equation in the time harmonic low frequency
        case and the adopted formulation of this problem.  Given the bounded domain $\Omega\subset\R^{3}$, we recall the time-harmonic Maxwell equations:
\begin{equation*}
\begin{array}{lllr}
   & \curl \bmH + (\mi \omega \epsilon - \sigma) \bmE & = \bmJ  & \text{ in } \Omega,\\
   & \curl \bmE - \mi \omega \mu \bmH & = 0  & \text{ in } \Omega,
\end{array}
\end{equation*}
and on the boundary $\partial\Omega$ we impose a magnetic boundary condition
$\bmH\times \bmn = 0$, where $\bmn$ stands for the outward normal to the boundary $\partial\Omega$. Here $\bmH$ and $\bmE$ denotes the magnetic and electric
fields, respectively. $\bmJ$ is the applied current density, $\epsilon$ is the
electric permittivity, $\mu$ is the magnetic permeability and $\sigma$ is the
electric conductivity.
In our case, the applied current density has support strictly included in the
insulator $\Omega_{t} $(interior of the tube).
 By neglecting the displacement current term, we formally obtain the eddy-current model, which reads:
\begin{equation}\label{eq:eddy3d}
\begin{array}{lllr}
& \curl \bmH &= \sigma \bmE + \bmJ   &\text{ in } \Omega,  \\
& \mu \bmH  &=\dfrac{1}{\mi \omega}\curl \bmE   &\text{ in } \Omega,
\end{array}
\end{equation}
We refer to the monograph~\cite{MR2680968} for an extensive overview of
        eddy-current models and formulations. 
	In this paper we adopt a potential formulation in which we look for the magnetic vector potential $\bmA$ and electric scalar potential $V$ (only defined on $\Omega_{\mathcal C}$) that satisfies
\begin{equation}\label{ChF}
\left\{\begin{array}{rll}
\mu \bmH&=\curl \bmA & \text{ in }\Omega, \\
\curl \bmH &= \sigma \bmE + \bmJ   &\text{ in } \Omega,\\
\bmE  &= \mi\omega \bmA + \grad V &\text{ in } \Omega_{\mathcal C}, \\
\divv\bmA&=0 &\text{ in } \Omega.\\
\bmn\times\dfrac{1}{\mu}\curl \bmA&=0 &\text{ on }\partial\Omega,\\
\bmA\cdot\bmn &= 0 &\text{ on }\partial\Omega,
\end{array}\right.
\end{equation}
where the equation \eqref{ChF}$_4$ stands for the Coulomb gauge condition. The
boundary condition~\eqref{ChF}$_5$ stands for the magnetic boundary condition
and the boundary condition~\eqref{ChF}$_6$ is equivalent to
$\epsilon\bmE\cdot\bmn=0$. The electric scalar potential $V$ is determined up to an additive constant in each connected-component of $\Omega_{\mathcal C}$, which has a connected boundary.

	

	  Notice that from Maxwell-Amp\`{e}re equation \eqref{eq:eddy3d}$_1$ we get
\begin{equation}\label{AWeakTMP}
\curl (\mu^{-1} \curl \bmA) - \sigma(\mi\omega \bmA+ \grad V) =
        \bmJ \mbox{ in }\Omega.
\end{equation}
In the following, the space $H(\curl;\Omega)$ indicates the set of real or
complex valued functions ${\bf v}\in(L^2(\Omega))^3$ such that $\curl {\bf
  v}\in(L^2(\Omega))^3$ and define
$$
\mathcal X(\Omega):=\{ {\bf v}\in H(\curl,\Omega), \divv {\bf v} = 0 \text{ in } \Omega, {\bf v}\cdot \bmn=0 \text{ on } \partial\Omega\}.
$$
For a vector magnetic potential $\bmA\in\mathcal X(\Omega)$, an
electric scalar potential $V\in H^1\left(\Omega_{\mathcal{C}}\right)\slash\C$
(the quotient by constants is relative to each connected component separately) 
and a test function $\bmPsi\in\mathcal X(\Omega)$ the weak formulation of
\eqref{AWeakTMP} reads
\begin{equation}\label{ampereWeakTMP}
\int_{\Omega} \dfrac{1}{\mu} \curl \bmA\cdot \curl \overline\bmPsi \d x - \int_{\Omega_{\mathcal C}}\sigma\left( \mi \omega \bmA\cdot \overline\bmPsi + \grad V\cdot \overline\bmPsi\right) \d x = \int_{\Omega} \bmJ\cdot \overline\bmPsi.
\end{equation}
Moreover, for any test function $\varPhi\in H^1\left(\Omega_{\mathcal{C}}\right)\slash\C$ the weak formulation of the necessary condition $-\divv\left(\sigma\bmE\right)=\divv \bmJ$ writes
$- \int_{\Omega_{\mathcal C}}\sigma\bmE\cdot \grad \overline{\varPhi} \d s=\int_{\Omega_{\mathcal C}} \bmJ\cdot \grad \overline{\varPhi} \d s $. Therefore using~\eqref{ChF}$_3$ we obtain:
\begin{equation}\label{divWeak}
-\sigma \int_{\Omega_{\mathcal C}} \left( \mi \omega \bmA + \grad V\right) \cdot \grad \overline\varPhi \d x = \int_{\Omega_{\mathcal C}} \bmJ\cdot \grad \overline\varPhi \d x.
\end{equation}
Following~\cite[Chp-6]{MR2680968} (and references therein), by introducing a constant
$\mu_{*}$, representing a suitable average of $\mu$ in $\Omega$, the Coulomb
gauge condition~\eqref{ChF}$_3$  can be incorporated in equation
\eqref{ampereWeakTMP} in the following way
\begin{equation}\label{ampereWeak}
\int_{\Omega} \dfrac{1}{\mu} \curl \bmA\cdot \curl \overline\bmPsi \d x + \dfrac{1}{ \mu_{*}}\int_{\Omega}  \divv \bmA \divv \overline\bmPsi\d x - \int_{\Omega_{\mathcal C}}\sigma\left( \mi \omega \bmA\cdot \overline\bmPsi + \grad V\cdot \overline\bmPsi\right) \d x = \int_{\Omega} \bmJ\cdot \overline\bmPsi.
\end{equation}
and the variational space $\mathcal X(\Omega)$ would then be replaced by
$\mathcal H(\Omega):=H(\curl,\Omega)\cap H_0(\divv,\Omega)$ or equivalently by
$H^1(\Omega)^3$ since the domain $\Omega$ is convex and sufficiently regular. Indeed, $\divv \bmA=0$ is verified in the weak sense.
Combining equations~\eqref{ampereWeak} with \eqref{divWeak} we can obtain a symmetric variational formulation as follows
\begin{align}\label{eq:fvEddy3dPotential}
  \mathcal{S}(\bmA,V;\bmPsi,\Phi) = \int_{\Omega} \bmJ\cdot \overline{\bmPsi} \d x
    - \dfrac{1}{\mi\omega} \int_{\Omega_{\mathcal{C}}} \bmJ \cdot \grad \overline{\Phi} \d x
  \qquad \forall \,(\bmPsi, \Phi) \in \X,
\end{align}
where the sesquilinear form $  \mathcal{S}$ is defined by:
\begin{equation}\label{eq:LHSfvDerivative3d}
\begin{array}{lll}
  \mathcal{S}(\bmA,V;\bmPsi,\Phi) &:=
    &\displaystyle\int_{\Omega}\left(\dfrac{1}{\mu} \curl \bmA \cdot \curl \overline{\bmPsi}  + \dfrac{1}{\mu_{*}} \divv \bmA \divv \overline{\bmPsi}\right)\d x \\
   &&+ \displaystyle\dfrac{1}{\mi\omega}\int_{\Omega_{\mathcal{C}}} \sigma(\mi\omega \bmA + \grad V)\cdot (\overline{\mi\omega \bmPsi + \grad \Phi}) \d x.
\end{array}
\end{equation}
The coercivity of $\mathcal S $ on   $H^1(\Omega)^3 \times
H^1\left(\Omega_{\mathcal{C}}\right)\slash\C$ (see for instance~\cite[Chp-6]{MR2680968}) ensures the well-posedness of the problem.

	This paper is organized as follows: after this introduction, we state
        in Section \ref{sec:statements} the nonlinear shape optimization
        problem by the introduction of the misfit function, which depends on
        the shape of the defect and in particular its eddy-current signal
        response. We derive, in Section~\ref{sec:adj}, the adjoint problem
        which is based on the shape derivative of the misfit function. At the
        end of this Section we explicitly formulate the shape gradient via the
        adjoint problem. In Section~\ref{sec:num}, we present and explain the
        algorithm of steepest descent based on  the use of fixed predefined
        grid. With
        Section~\ref{sec:exp}, we conclude the paper with numerical experiments
        that illustrate the robustness of the method. Some technical materials
        related to shape derivative are reported in the appendix for the readers' convenience.
	
	We conclude this section with the introduction of some useful notations. We denote by $[\cdot]$ the jump across the interface $\Gamma$:
	$
	\left[ F \right]  = \lim_{t\searrow 0}F(x+t\bmn) - \lim_{t\searrow 0}F(x-t\bmn)\quad \forall x\in\Gamma
	$, we recall here that $\bmn$ denote the normal to $\Gamma$ pointing outside $\Omega_d$. 
For any vector $\bmA$ and differentiable scalar $V$,
we respectively denote the tangential component and the tangential gradient  on
some boundary or interface having a normal $\bmn$ by $\bmA_{\tau} := \bmA - (\bmA\cdot \bmn)\bmn$ and $\grad_{\tau}V := \grad V  - \partial_{\bmn}V\cdot\bmn$.
	We finally shall use the notation
$
  \X(\Omega) :=  \mathcal H(\Omega)\times H^1\left(\Omega_{\mathcal{C}}\right)\slash\C.
$

\section{Statement of the inverse problem}\label{sec:statements}

\subsection{Impedance measurements}
The deposit probing is an operation of scan with two coils introduced inside
the tube along its axis from a vertical position $\zeta_{\min}$ to a vertical
position $\zeta_{\max}$. 
At each position $\zeta \in [\zeta_{\min}, \zeta_{\max}]$, we measure the impedance signal $Z^{}(\zeta)$. According to \cite[(10a)]{auld1999review},
in the full Maxwell's system, the impedance measured in the coil $k$ when the electromagnetic field is induced by the coil $l$ writes
$
  \triangle Z_{kl} = \dfrac{1}{|J|^{2}}\int_{\partial \Omega_{d}}
    (\bmE^{0}_{l}\times \bmH_{k} - \bmE_{k}\times \bmH^{0}_{l})\cdot \bmn \d S,
$
 where $\bmE^{0}_{l}$ and $\bmH^{0}_{l}$ are respectively the electric field and
the magnetic field in the deposit-free case with corresponding permeability and
conductivity distributions $\mu^{0}$, $\sigma^{0}$, while $\bmE_{k}$,
$\bmH_{k}$ are those in the case with deposits.
Using the divergence theorem, we obtain the following volume representation of
the impedances
\begin{align}\label{Chap2_eq:impedance3d}
  \triangle Z_{kl} & = \dfrac{1}{|J|^{2}}\int_{\Omega_{d}}
    \divv(\bmE^{0}_{l}\times \bmH_{k} - \bmE_{k}\times \bmH^{0}_{l}) \d x \notag \\
  & = \dfrac{1}{|J|^{2}}\int_{\Omega_{d}}
    (\curl\bmE^{0}_{l} \cdot\bmH_{k} - \bmE^{0}_{l}\cdot \curl\bmH_{k}
    - \curl\bmE_{k} \cdot\bmH^{0}_{l} + \bmE_{k}\cdot \curl\bmH^{0}_{l}) \d x \notag \\
  & = \dfrac{1}{\mi\omega |J|^{2}}\int_{\Omega_{d}}
    \left(  (\dfrac{1}{\mu}-\dfrac{1}{\mu^{0}} )\curl\bmE_{k}\cdot \curl\bmE^{0}_{l}
    - \mi\omega  (\sigma-\sigma^{0} )\bmE_{k}\cdot\bmE^{0}_{l}\right) \d x.
\end{align}
In the last equality we used the eddy-current
model~\eqref{eq:eddy3d}. Furthermore, using the relation $\bmE  = \mi\omega
\bmA + \grad V$ we replace the electric field $\bmE$ by the vector potential
$\bmA$ and we thus obtain the following shape dependent impedance measurement
formula\footnote{
Let's recall the fact that $\sigma^0$ is an $\overline\epsilon$-conductivity. Hence the electric field $\bmE^0_l$ has a sense with $\bmE^0_l=i\omega\bmA^0_l+\grad V^0_l$.
  }
\begin{align}\label{eq:impedance3d}
  \triangle Z_{kl}(\Omega_{d})\\
    = \dfrac{\mi\omega}{|J|^{2}} \int_{\Omega_{d}} \bigg( & (\dfrac{1}{\mu}-\dfrac{1}{\mu^{0}})\curl \bmA_{k}\cdot\curl\bmA^{0}_{l}
   - \dfrac{1}{\mi\omega} (\sigma-\sigma^{0} )(\mi\omega \bmA_{k}+\grad V_{k})\cdot(\mi\omega \bmA^{0}_{l} + \grad V^{0}_{l}) \bigg) \d x.\nonumber
\end{align}
\subsection{A least squares formulation}
Let us denote by $Z^{\natural}(\Omega_d^\star;\zeta)$ the impedance response signals of a probed deposit
$\Omega_d^\star$ that we would like to estimate.  We shall use
the shape dependent form in the impedance signal response in order to convert
the signal anomaly to a shape perturbation. This inverse problem
will be solved by minimizing a least square misfit function
representing the error between computed and observed signals integrated over
the coil positions. This misfit function is defined as follows:
\begin{align}\label{eq:leastSquare}
\mathbf{f}(\Omega_{d}) = \int_{\zeta_{\min}}^{\zeta_{\max}} |Z(\Omega_{d};\zeta) - Z^{\natural}(\Omega_d^\star;\zeta)|^{2} \d\zeta,
\end{align}
where $Z$ is either $Z_{FA}$ or $Z_{F3}$ according to the measurement mode used
in practice:
  $Z_{FA}(\Omega_{d}) := \dfrac{\mi}{2}(\triangle Z_{11}(\Omega_{d})+\triangle Z_{21}(\Omega_{d}))$,
 or $Z_{F3}(\Omega_{d}):= \dfrac{\mi}{2}(\triangle Z_{11}(\Omega_{d})-\triangle
 Z_{22}(\Omega_{d}))$.
Minimizing this functional using a steepest descent method requires a
characterization of its derivative with respect to perturbations of $\Omega_{d}$. This is the objective of next section. 

\section{Adjoint problem and explicit formulation of the shape gradient}\label{sec:adj}
We shall first study the shape derivative of the solution $(\bmA,V)$ with
respect to deformations of the deposit shape. This derivative will then allow us
to obtain an expression of the cost-functional derivative. A computable version
of this derivative is then derived through the introduction of an adjoint
state. 
\subsection{A preliminary result on the material derivative}

In this part, we formally derive the expression of the material derivative of the solution to the eddy-current model
on a regular open set with constant physical coefficients $\mu$, $\sigma$. This
result will be used in next sections to obtain the material derivative of the
eddy-current model with piecewise constant coefficients as well as the shape
derivative of the impedance measurements. We begin by introducing the shape and
material derivatives \cite[Section 6.3.3]{de2006conception}. For any regular open set $\Omega \subset \R^{3}$, we consider a domain deformation as a perturbation of the identity
$
  \Id + \bmtheta : \Omega  \rightarrow\Omega_{\theta} \notag,
  x  \mapsto y,
$
where $\bmtheta\in \mathcal{C}:=(C^{2}(\R^{3};\R^{3}))^{3}$ is a small perturbation of the domain.
To make a difference between the differential operators before and after the variable substitution,
we denote by $\curl_{x}$, $\divv_{x}$, $\grad_{x}$ the curl, divergence and gradient operators on
$\Omega$ with $x$-coordinates, and respectively by $\curl_{y}$, $\divv_{y}$, $\grad_{y}$ those on
$\Omega_{\theta}$ with $y$-coordinates. For any $(\bmA(\Omega_{\theta}), V(\Omega_{\theta}))$
defined on $\Omega_{\theta}$, we set
  \begin{align*}
  &\bmA_{\curl}(\bmtheta) := (I+\grad\bmtheta)^{t} \bmA(\Omega_{\theta})\circ(\Id + \bmtheta), \\
  &\bmA_{\divv}(\bmtheta) := \det(I+ \grad\bmtheta)(I+ \grad\bmtheta)^{-1} \bmA(\Omega_{\theta})\circ(\Id + \bmtheta), \\
  & V_{\grad}(\bmtheta) := V(\Omega_{\theta})\circ(\Id + \bmtheta).
  \end{align*}
These quantities conserve the corresponding differential operators in the following sense
(see for example \cite[(3.75), Corollary 3.58, Lemma 3.59]{MR2059447})
\begin{equation}\label{eq:DifChangeVariable}
  \begin{array}{lll}
  \dfrac{I+\grad\bmtheta}{\det(I+\grad\bmtheta)}\curl_{x} \bmA_{\curl}(\bmtheta)
    & = (\curl_{y} \bmA(\Omega_{\theta}))\circ (\Id + \bmtheta), \vspace{0.07in}\\
  \dfrac{1}{\det(I+\grad\bmtheta)}\divv_{x} \bmA_{\divv}(\bmtheta)
    & = (\divv_{y} \bmA(\Omega_{\theta}))\circ(\Id + \bmtheta), \vspace{0.07in}\\
  (I+\grad\bmtheta)^{-t} \grad_{x}V_{\grad}(\bmtheta)
    & = (\grad_{y} V(\Omega_{\theta}))\circ(\Id + \bmtheta),
  \end{array}
\end{equation}
where $\grad\bmtheta := (\dfrac{\partial \theta_{i}}{\partial x_{j}})_{i,j}$ is the Jacobian matrix.
\\	
In order to simplify the notation we use $\curl$, $\divv$ and $\grad$
        for respectively $\curl_{x}$, $\divv_{x}$ and $\grad_{x}$.
\\		
Let $(\bmA(\Omega), V(\Omega))$ be some shape-dependent functions that belong to some Banach space $\mathcal{W}(\Omega)$,\
  and $\bmtheta \in \mathcal C$ a shape perturbation.
  The material derivatives $(\bmB(\bmtheta),U(\bmtheta))$ of $(\bmA,V)$,
  if they exist, are defined as
\begin{equation}\label{defi:derivatives3d}
  \begin{cases}
    \bmA_{\curl}(\bmtheta)  & = \bmA_{\curl}(0) + \bmB(\bmtheta) + o(\bmtheta)
      = \bmA(\Omega) + \bmB(\bmtheta) + o(\bmtheta),\\ 
    V_{\grad}(\bmtheta)  & = V_{\grad}(0) + U(\bmtheta) + o(\bmtheta)
      = V(\Omega) + U(\bmtheta) + o(\bmtheta), \\ 
  \end{cases}\end{equation}
 We also define the shape derivatives $(\bmA'(\bmtheta),V'(\bmtheta))$ of $(\bmA,V)$ by
\begin{equation}\label{eq:dAdVmaterial}
 \begin{cases}
    \bmA'(\bmtheta) & := \bmB(\bmtheta) - (\bmtheta\cdot\grad)\bmA(\Omega) - (\grad\bmtheta)^{t}\bmA(\Omega),\\
    V'(\bmtheta) & := U(\bmtheta) - \bmtheta\cdot \grad V(\Omega).
  \end{cases}\end{equation}
  The derivative $\bmB_{\divv}(\bmtheta)$ of $\bmA$ which conserve the divergence operator is given by
  \begin{align}
    \bmB_{\divv}(\bmtheta) & := \bmB(\bmtheta) + (\divv\bmtheta I - \grad\bmtheta - (\grad\bmtheta)^{t}) \bmA(\Omega).
      \label{eq:dAdivv_dAmaterial}
  \end{align}
  Using the chain rule, in any open set of 
  $\Omega\cap\Omega_{\theta}$ we formally have
  \begin{align}
    \bmA(\Omega_{\theta}) & = \bmA(\Omega) + \bmA'(\bmtheta) + o(\bmtheta), \label{eq:dAshape}\\
    \bmA_{\divv}(\bmtheta) & = \bmA(\Omega) + \bmB_{\divv}(\bmtheta) + o(\bmtheta), \label{eq:dAdivv} \\
    V(\Omega_{\theta})      & = V(\Omega) + V'(\bmtheta) + o(\bmtheta). \label{eq:dVshape}
  \end{align}
To ease further discussions, in particular the derivation of the variational formulation
\eqref{eq:fvDerivative3dBis} from \eqref{eq:fvDerivative3d}, we give a preliminary result.
Assume that the coefficients $\mu$ and $\sigma$ are constant on $\Omega$.
We set a shape-dependent form
\begin{align}\label{eq:formAlpha3d}
  \sfa(\Omega)\big(\bmA,V;\bmPsi,\Phi \big)
    := \int_{\Omega}\dfrac{1}{\mu}\curl \bmA \cdot \curl \overline{\bmPsi} \d x
    + \dfrac{1}{\mi\omega}\int_{\Omega} \sigma(\mi\omega\bmA + \grad V)\cdot (\overline{\mi\omega\bmPsi + \grad\Phi}) \d x.
\end{align}
Compared to the variational form $\mathcal{S}$ defined in \eqref{eq:fvEddy3dPotential},
the above form $\sfa(\Omega)$ get rid of the penalization term
$\int_{\Omega}(\mu^{*})^{-1}\divv\bmA\divv\overline{\bmPsi}\d x$.
\begin{lemma}\label{Chap5_lemm:compute3d}
  Let $\Omega$ be a regular open set, $\mu>0$ and $\sigma\ge 0$ constant on $\Omega$ and
  $\Id + \bmtheta: \Omega \rightarrow \Omega_{\theta}$ a given deformation.
  Let $(\bmA,V) = (\bmA(\Omega), V(\Omega))$ and $(\bmPsi, \Phi) = (\bmPsi(\Omega), \Phi(\Omega))$
  be some shape-dependent functions with sufficient regularity.
  We assume that the material derivatives $(\bmB(\bmtheta), U(\bmtheta))$ of $(\bmA, V)$,
  the shape derivatives $(\bmA'(\bmtheta), V'(\bmtheta))$ of $(\bmA, V)$
  and the material derivatives $(\bmeta(\bmtheta), \chi(\bmtheta))$ of $(\bmPsi, \Phi)$
  defined with \eqref{defi:derivatives3d} exist. If $(\bmA(\Omega), V(\Omega))$
  satisfy in the weak sense
  \begin{equation}\label{eq:lemmCompute}
    \begin{cases}
      \,\curl(\mu^{-1}\curl \bmA) - \sigma(\mi\omega\bmA + \grad V)  = 0 & \text{in }\Omega,\\
      \divv \bmA = 0 & \text{in }\Omega,\\
      \sigma(\mi\omega\bmA + \grad V)\cdot \bmn = 0 & \text{on }\partial\Omega,\\
    \end{cases}
  \end{equation}
  then the shape derivative of $\sfa(\Omega)$ that we denote by $\sfaprime(\Omega)(\bmtheta)$, i.e.
$
    \sfa(\Omega_{\theta})\big(\bmA, V;\bmPsi,\Phi\big) 
      = \sfa(\Omega)\big(\bmA, V;\bmPsi,\Phi\big)
      + \sfaprime(\Omega)(\bmtheta)\big(\bmA, V;\bmPsi,\Phi\big) + o(\bmtheta),
$
 satisfies
  \begin{equation}\label{eq:alphaDerivative}
  \begin{array}{ll}
     \sfaprime(\Omega)(\bmtheta)\big(\bmA,V;\bmPsi,\Phi\big)
      =& \sfa(\Omega)\big(\bmA'(\bmtheta),V'(\bmtheta);\bmPsi,\Phi\big)
      + \sfa(\Omega)\big(\bmA,V;\bmeta(\bmtheta),\chi(\bmtheta)\big)  \\
    & + \displaystyle\int_{\partial\Omega}\dfrac{1}{\mu}(\bmtheta\cdot \curl\bmA)(\bmn\cdot \curl\overline{\bmPsi}) \d s\\
    & + \displaystyle\dfrac{1}{\mi\omega}\int_{\partial\Omega} \sigma(\bmn\cdot\bmtheta)(\mi\omega\bmA_{\tau} + \grad_{\tau} V)
      \cdot(\overline{\mi\omega\bmPsi_{\tau}+\grad_{\tau} \Phi}) \d s.
  \end{array}\end{equation}
\end{lemma}
The proof of this Lemma is given in the Appendix.
\subsection{Material derivative of the solution to the eddy-current problem}
In this part, we show the existence of the material derivative of the solution to the eddy-current problem
with respect to a domain variation,
and give its weak formulation with a right hand side in the form of some boundary integrals.
We rewrite the variational formulation of the eddy-current model \eqref{eq:fvEddy3dPotential} on $\Omega_{\theta}$. For any test functions $(\bmPsi, \Phi) \in \X$
\begin{equation}\label{fvEddy3dPotentialTheta}
\begin{array}{ll}
   \mathcal{S}(\bmA(\Omega_{\theta}), V(\Omega_{\theta});\bmPsi(\Omega_{\theta}),\Phi(\Omega_{\theta}))
  & = \displaystyle\int_{\Omega_{\theta}} \left(\dfrac{1}{\mu} \curl_{y} \bmA \cdot \curl_{y}\overline{\bmPsi}
    + \dfrac{1}{\mu_{*}} \divv_{y} \bmA \divv_{y} \overline{\bmPsi}\right) \d y\\
   &\qquad + \dfrac{1}{\mi\omega}\displaystyle\int_{\Omega_{\mathcal{C}\theta}} \sigma(\mi\omega \bmA + \grad V)\cdot (\overline{\mi\omega \bmPsi + \grad \Phi}) \d y\\
  & = \displaystyle\int_{\Omega_{\theta}} \bmJ\cdot \overline{\bmPsi} \d y - \dfrac{1}{\mi\omega} \displaystyle\int_{\Omega_{\mathcal{C}\theta}} \bmJ \cdot \grad \overline{\Phi} \d y.
\end{array}
\end{equation}
We choose the test functions as follows (so that their material derivatives vanish)
$
  \bmvarPsi = (I+\grad\bmtheta)^{t} \bmPsi(\Omega_{\theta})\circ(\Id + \bmtheta),
  \varPhi = \Phi(\Omega_{\theta})\circ(\Id + \bmtheta).
$
Since the supports of $\bmJ$ and $\bmtheta$ are disjoint, i.e.
$\supp (\bmJ) \cap \supp(\bmtheta) = \emptyset$,
the right-hand side of the weak formulation \eqref{fvEddy3dPotentialTheta} writes simply:
\begin{align}\label{eq:fvEddy3dRHS}
  \int_{\Omega} \bmJ\cdot \overline{\bmvarPsi} \d x
  - \dfrac{1}{\mi\omega} \int_{\Omega_{\mathcal{C}}} \bmJ \cdot \grad \overline{\varPhi} \d x.
\end{align}
We consider the following term which conserves the divergence operator
\begin{align*}
  & \bmPsi_{\divv}(\bmtheta) := \det(I+ \grad\bmtheta)(I+ \grad\bmtheta)^{-1}(I+ \grad\bmtheta)^{-t} \bmvarPsi
  \\ &\hspace{1.5cm} = \det(I+ \grad\bmtheta)(I+ \grad\bmtheta)^{-1} \bmPsi(\Omega_{\theta})\circ(\Id + \bmtheta), \notag \\
  & \divv \bmPsi_{\divv}(\bmtheta)
    = \det(I+\grad\bmtheta)\big(\divv_{y} \bmPsi(\Omega_{\theta})\big)\circ(\Id + \bmtheta).
\end{align*}
By variable substitution $y=(\Id+\bmtheta)x$, the left hand side of
\eqref{fvEddy3dPotentialTheta} can be written as
\begin{equation}\label{eq:fvEddy3dLHS}
\begin{array}{ll}
  & \displaystyle\int_{\Omega} \bigg(\dfrac{1}{\mu}\dfrac{(I+\grad\bmtheta)^{t}(I+\grad\bmtheta)}{|\det(I+\grad\bmtheta)|}
      \curl\bmA_{\curl} \cdot \curl \overline{\bmvarPsi}
    + \dfrac{1}{\mu_{*}}\dfrac{1}{|\det(I+\grad\bmtheta)|}
      \divv \bmA_{\divv} \divv \overline{\bmPsi_{\divv}} \bigg)\d x \\
  &+ \dfrac{1}{\mi\omega}\displaystyle\int_{\Omega_{\mathcal{C}}}\sigma |\det(I+\grad\bmtheta)|(I+\grad\bmtheta)^{-1}(I+\grad\bmtheta)^{-t}
     \big(\mi\omega\bmA_{\curl} + \grad V_{\grad}\big)
     \cdot \big(\overline{\mi\omega\bmvarPsi + \grad\varPhi}\big) \d x.
\end{array}
\end{equation}

\begin{theorem}\label{theo:shapeContinuity}
  Let $\bmtheta\in \mathcal C$ a domain perturbation.
  Let $\mu > 0$, $\sigma \ge 0$ belong to $L^{\infty}(\Omega)$. We recall that
  $\bmJ \in L^{2}(\Omega)^{3}$ has compact support in $\Omega_{s} \subset \Omega_{v}$
  and satisfies $\divv \bmJ = 0$ in $\Omega_{s}$ and $\supp (\bmJ) \cap \supp(\bmtheta) = \emptyset$.
  If $(\bmA(\Omega),V(\Omega)) =(\bmA_{\curl}(0),V_{\grad}(0))$ is the solution to the eddy-current problem
  \eqref{eq:fvEddy3dPotential} and $(\bmA(\Omega_{\bmtheta}), V(\Omega_{\bmtheta}))
  = \big((I+\grad\bmtheta)^{-t}\bmA_{\curl}(\bmtheta)\circ(\Id+\bmtheta)^{-1}, V_{\grad}(\bmtheta)\circ(\Id+\bmtheta)^{-1}\big)$
  the solution to the problem \eqref{fvEddy3dPotentialTheta}, then
  \begin{align*}
    \lim_{\|\bmtheta\|_{\mathcal{C}} \to 0}
      \Big\|\big(\bmA_{\curl}(\bmtheta) - \bmA_{\curl}(0),
      V_{\grad}(\bmtheta)-V_{\grad}(0)\big)\Big\|_{\X} = 0.
  \end{align*}
\end{theorem}
\begin{proof}
  We recall that $\bmA(\Omega)$ and $\bmA(\Omega_{\theta})$ satisfy the Coulomb gauge condition
  on $\Omega$ and on $\Omega_{\theta}$ respectively: $\divv \bmA(\Omega) = 0$ on $\Omega$,
  $\divv_{y} \bmA(\Omega_{\theta}) = 0$ on $\Omega_{\theta}$.
  From the weak formulations \eqref{eq:fvEddy3dPotential}, \eqref{fvEddy3dPotentialTheta}
  the identities \eqref{eq:fvEddy3dRHS}, \eqref{eq:fvEddy3dLHS}
  and the developments in \eqref{eq:thetaDevelop} we obtain
  \begin{equation}\label{eq:shapeContinuity}
  \begin{array}{ll}
    & \mathcal{S}\Big(\bmA_{\curl}(\bmtheta) - \bmA_{\curl}(0), V_{\grad}(\bmtheta)-V_{\grad}(0);\bmvarPsi, \varPhi\Big) \\
    & = \displaystyle\int_{\Omega} \dfrac{1}{\mu}(\divv\bmtheta I - \grad\bmtheta -(\grad\bmtheta)^{t})
      \curl \bmA_{\curl}(\bmtheta) \cdot \curl\overline{\bmvarPsi} \d x  \\
    & \quad + \dfrac{1}{\mi\omega}\displaystyle\int_{\Omega_{\mathcal{C}}} \sigma
      (-\divv\bmtheta I + \grad\bmtheta + (\grad\bmtheta)^{t})(\mi\omega \bmA_{\curl}(\bmtheta)+\grad V_{\grad}(\bmtheta))
      \cdot(\overline{\mi\omega\bmvarPsi + \grad\varPhi})\d x
      + o(\bmtheta).
  \end{array}
  \end{equation}
  Obviously the right hand side of the above equality goes to zero as $\|\bmtheta\|_{\mathcal{C}} \to 0$.
  Since the form $\mathcal{S}$ is coercive (see \cite[Section 6.1.2]{MR2680968}),
  this implies  $\big\|\big(\bmA_{\curl}(\bmtheta) - \bmA_{\curl}(0),
  V_{\grad}(\bmtheta)-V_{\grad}(0)\big)\big\|_{\X} \to 0$
  as $\|\bmtheta\|_{\mathcal C} \to 0$.
\end{proof}

\begin{theorem}\label{theo:differentiability}
  Under the same assumptions as in Theorem \ref{theo:shapeContinuity},
  the material derivative of the solution $(\bmA(\Omega), V(\Omega))$ to the eddy-current problem
  \eqref{eq:fvEddy3dPotential} with respect to a domain variation $\Id+\bmtheta$ exists.
  If it is denoted by $(\bmB(\bmtheta),U(\bmtheta))$, then
  \begin{align*}
    \lim_{\|\bmtheta\|_{\mathcal C}\to 0} \dfrac{1}{\|\bmtheta\|_{\mathcal C}}
      \Big\|\big(\bmA_{\curl}(\bmtheta) - \bmA_{\curl}(0) - \bmB(\bmtheta),
      V_{\grad}(\bmtheta)-V_{\grad}(0)-U(\bmtheta)\big)\Big\|_{\X} = 0.
  \end{align*}
\end{theorem}
\begin{proof}
Let $(\bmB(\bmtheta), U(\bmtheta))$ the unique solution in $\X$
to the weak formulation
 \begin{equation}\label{eq:fvDerivative3d}
  \mathcal{S}(\bmB(\bmtheta),U(\bmtheta);\bmvarPsi,\varPhi)
    = L(\bmvarPsi,\varPhi) \qquad
    \forall (\bmvarPsi,\varPhi) \in \X,
\end{equation}
where
\begin{equation} \label{eq:RHSfvDerivative3d}
\begin{array}{ll}
  L(\bmvarPsi,\varPhi) := & \int_{\Omega}
    \dfrac{1}{\mu}(\divv\bmtheta I - \grad\bmtheta -(\grad\bmtheta)^{t})\curl \bmA \cdot \curl\overline{\bmvarPsi}  \d x \\
  &-  \displaystyle\int_{\Omega_{d}}\dfrac{1}{\mu_{*}}
    \divv\bigg((\divv\bmtheta I - \grad\bmtheta - (\grad\bmtheta)^{t}) \bmA\bigg)
    \cdot \divv\overline{\bmvarPsi}\d x \\
  & +  \displaystyle \dfrac{1}{\mi\omega}\int_{\Omega_{\mathcal{C}}}
    \sigma (-\divv\bmtheta I + \grad\bmtheta + (\grad\bmtheta)^{t})(\mi\omega \bmA + \grad V)
    \cdot (\overline{\mi\omega\bmvarPsi + \grad\varPhi})\d y.
\end{array}\end{equation}
Let $\bmB_{\divv}(\bmtheta)$ defined by \eqref{eq:dAdivv_dAmaterial}.
Then we can rewrite the weak formulation \eqref{eq:fvDerivative3d} as
\begin{equation}\label{eq:fvDerivative3dZero}
\begin{array}{cc}
 \displaystyle  \int_{\Omega}\left(\dfrac{1}{\mu} \curl \bmB(\bmtheta)\cdot \curl\overline{\bmvarPsi}
    + \dfrac{1}{\mu_{*}}\divv \bmB_{\divv}(\bmtheta) \divv\overline{\bmvarPsi}\right) \d x
    \\ +\hspace{1.5cm} \dfrac{1}{\mi\omega}\int_{\Omega_{\mathcal{C}}} \sigma
    (\mi\omega \bmB(\bmtheta) + \grad U(\bmtheta))\cdot (\overline{\mi\omega\bmvarPsi + \grad \varPhi}) \d x \\
   = \\
  \displaystyle\int_{\Omega} \dfrac{1}{\mu}(\divv\bmtheta I - \grad\bmtheta -(\grad\bmtheta)^{t})
    \curl \bmA \cdot \curl\overline{\bmvarPsi} \d x    
    \\ +\hspace{1.5cm} \dfrac{1}{\mi\omega}\int_{\Omega_{\mathcal{C}}} \sigma
    (-\divv\bmtheta I + \grad\bmtheta + (\grad\bmtheta)^{t})(\mi\omega \bmA+\grad V)\cdot(\overline{\mi\omega\bmvarPsi + \grad\varPhi})\d x.
\end{array}
\end{equation}
From \eqref{eq:dAdivv} and the Coulomb gauge conditions satisfied by $\bmA(\Omega)$ and $\bmA(\Omega_{\theta})$
we deduce that $\divv \bmB_{\divv}(\bmtheta) = o(\bmtheta)$.
Considering the fact that $(\bmA,V) = (\bmA_{\curl}(0), V_{\grad}(0))$, \eqref{eq:shapeContinuity} and \eqref{eq:fvDerivative3dZero} yield
\begin{align*}
  & \mathcal{S}\Big(\bmA_{\curl}(\bmtheta) - \bmA_{\curl}(0) - \bmB(\bmtheta),
    V_{\grad}(\bmtheta) - V_{\grad}(0) - U(\bmtheta);\bmvarPsi, \varPhi\Big) + o(\bmtheta) \notag \\
  & = \int_{\Omega} \dfrac{1}{\mu}(\divv\bmtheta I - \grad\bmtheta -(\grad\bmtheta)^{t})
    \big(\curl\bmA_{\curl}(\bmtheta) - \curl\bmA_{\curl}(0)\big) \cdot \curl\overline{\bmvarPsi} \d x \notag \\
  & \quad + \dfrac{1}{\mi\omega}\int_{\Omega_{\mathcal{C}}} \sigma
    (-\divv\bmtheta I + \grad\bmtheta + (\grad\bmtheta)^{t})
    \Big(\mi\omega (\bmA_{\curl}(\bmtheta) - \bmA_{\curl}(0)) + (\grad V_{\grad}(\bmtheta) - \grad V_{\grad}(0)\Big)
    \cdot(\overline{\mi\omega\bmvarPsi + \grad\varPhi})\d x.
\end{align*}
Theorem \ref{theo:shapeContinuity} implies that the right hand side of the above equality is of order $o(\bmtheta)$
as $\|\bmtheta \|_{\mathcal{C}}\to 0$. The coercivity of $\mathcal{S}$ ensures the result as stated.
\end{proof}
	
\begin{proposition}\label{prop:fvMateriel}
  Under the same assumptions as in Theorem \ref{theo:shapeContinuity},
  we assume in addition that $\mu$, $\sigma$ are piecewise constant and constant in each subdomain
  ($\Omega_{s}$, $\Omega_{t}$, $\Omega_{d}$, $\Omega_{v}$ or $\Omega_{p}$).
  If the domain perturbation $\bmtheta$ has support only on a vicinity of the interface $\Gamma$
  between the deposit domain $\Omega_{d}$ and the vacuum $\Omega_{v}$
  ($\Gamma = \overline{\Omega_d} \cap \overline{\partial\Omega_{v}}$)
  and vanishes in $\Omega_{s}$, then the material derivatives $(\bmB(\bmtheta), U(\bmtheta))$ of $(\bmA,V)$ satisfies
  \begin{align}\label{eq:fvDerivative3dBis}
    & \mathcal{S}(\bmB(\bmtheta),U(\bmtheta);\bmvarPsi,\varPhi) = \mathcal{L}(\bmvarPsi,\varPhi)
    \quad \forall (\bmvarPsi,\varPhi)\in \X,
  \end{align}
  where
  \begin{equation}\label{eq:RHSfvDerivative3dTris}
  \begin{array}{ll}
 &\mathcal{L}(\bmvarPsi,\varPhi)  : = \\
 & \displaystyle \int_{\Omega_{d}} \left(\dfrac{1}{\mu}     \curl((\bmtheta\cdot\grad)\bmA + (\grad\bmtheta)^{t}\bmA)\cdot\curl\overline{\bmvarPsi}     + \dfrac{1}{\mu_{*}}\divv((\bmtheta\cdot\grad)\bmA + (\grad\bmtheta)^{t}\bmA) \divv\overline{\bmvarPsi} \right)\d x \\
  &  + \dfrac{1}{\mi\omega}\displaystyle\int_{\Omega_{\mathcal{C}}}\sigma\bigg(\mi\omega\big((\bmtheta\cdot\grad)\bmA + (\grad\bmtheta)^{t}\bmA\big)
    + \grad(\bmtheta\cdot\grad V)\bigg)\cdot(\overline{\mi\omega\bmvarPsi+\grad\varPhi}) \d x \\
 & + \displaystyle\int_{\Gamma}\left[\dfrac{1}{\mu}\right](\bmtheta\cdot\bmn)
    (\bmn\cdot \curl \bmA) (\bmn\cdot \curl\overline{\bmvarPsi})\d s \\
 &   + \dfrac{1}{\mi\omega}\displaystyle \int_{\Gamma} (\bmtheta\cdot\bmn)[\sigma](\mi\omega \bmA_{\tau}
    + \grad_{\tau}V)\cdot(\overline{\mi\omega\bmvarPsi_{\tau}+\grad_{\tau}\varPhi})\d s.
\end{array}
  \end{equation}
\end{proposition}
\begin{proof}
Let $\Lambda := \{s,t,d,v,p\}$ a set of indices with its elements indicating the different sub-domains
as well as the corresponding permeabilities and conductivities.
We rewrite left-hand-side of the variational formulation \eqref{fvEddy3dPotentialTheta} as
\begin{align*}
  \mathcal{S}(\bmA(\Omega_{\theta}),V(\Omega_{\theta});\bmPsi(\Omega_{\theta}),\Phi(\Omega_{\theta}))
  =& \sum_{i\in\Lambda} \sfaI(\Omega_{i\theta})
    (\bmA,V;\bmPsi,\Phi)
\\&  + \int_{\Omega_{\theta}} \dfrac{1}{\mu_{*}}
    \divv_{y}\bmA(\Omega_{\theta})\cdot \divv_{y}\overline{\bmPsi(\Omega_{\theta})} \d y.
\end{align*}
According to the definition of the test functions $(\bmPsi(\Omega_{\theta}),\Phi(\Omega_{\theta}))$, their respective material derivatives vanish. Since $(\bmA(\Omega_{\theta}),V(\Omega_{\theta}))$ satisfy both~\eqref{ampereWeak}, we can apply Lemma \ref{Chap5_lemm:compute3d} to the terms $\sfaI(\Omega_{i\theta})$,
which yields the shape derivative
\begin{align}\label{eq:materialDifAterm}
  & \sum_{i\in\Lambda} \sfaprimeI(\Omega_{i})(\bmtheta)(\bmA,V;\bmvarPsi,\varPhi) \notag \\
  & = \sum_{i\in\Lambda} \sfaI(\Omega_{i})(\bmB(\bmtheta),U(\bmtheta);\bmvarPsi,\varPhi)
     + \sum_{i\in\Lambda} \sfaI(\Omega_{i})(-(\bmtheta\cdot\grad)\bmA - (\grad\theta)^{t}\bmA,- (\bmtheta\cdot\grad V);\bmvarPsi,\varPhi)
    \notag \\
  & \quad - \int_{\Gamma}\left[\dfrac{1}{\mu}
    (\bmtheta\cdot \curl \bmA) (\bmn\cdot \curl\overline{\bmvarPsi})\right]\d s
    - \dfrac{1}{\mi\omega} \int_{\Gamma} (\bmtheta\cdot\bmn)[\sigma](\mi\omega \bmA_{\tau}
    + \grad_{\tau}V)\cdot(\overline{\mi\omega\bmvarPsi_{\tau}+\grad_{\tau}\varPhi})\d s \notag \\
  & = \mathcal{S}(\bmB(\bmtheta),U(\bmtheta);\bmvarPsi,\varPhi)
    - \int_{\Omega}\dfrac{1}{\mu_{*}}\divv \bmB(\bmtheta)\divv\overline{\bmvarPsi} \d x
    \\ &\quad+ \sum_{i\in\Lambda} \sfaI(\Omega_{i})(-(\bmtheta\cdot\grad)\bmA - (\grad\theta)^{t}\bmA,- (\bmtheta\cdot\grad V);\bmvarPsi,\varPhi) \notag \\
  & \quad - \int_{\Gamma}\left[\dfrac{1}{\mu}\right](\bmtheta\cdot\bmn)
    (\bmn\cdot \curl \bmA) (\bmn\cdot \curl\overline{\bmvarPsi})\d s \notag
  \\ &\quad - \dfrac{1}{\mi\omega} \int_{\Gamma} (\bmtheta\cdot\bmn)[\sigma](\mi\omega \bmA_{\tau}
  + \grad_{\tau}V)\cdot(\overline{\mi\omega\bmvarPsi_{\tau}+\grad_{\tau}\varPhi})\d s.\notag
\end{align}
In the last equality we have used the transmission conditions
 $ [\bmn\cdot \curl\bmA] = [\bmn\times(\mu^{-1}\curl\bmA\times \bmn)] = 0 \quad \text{on } \Gamma$.
Using the identities \eqref{eq:VectorFormula} and the Coulomb gauge condition $\divv \bmA = 0$,
one verifies that on each subdomain $\Omega_{i}$ ($i\in \Lambda$) of $\Omega$
\begin{align}\label{eq:divDiv}
  \divv((\divv\bmtheta I - \grad\bmtheta - (\grad\bmtheta)^{t}) \bmA) = - \divv((\bmtheta\cdot\grad)\bmA+(\grad\bmtheta)^{t}\bmA).
\end{align}
From the derivation of $L(\bmvarPsi,\Phi)$ \eqref{eq:RHSfvDerivative3d} and the equality \eqref{eq:divDiv},
one easily  deduces that the shape derivative of the penalization term
  $\int_{\Omega_{\theta}} \dfrac{1}{\mu_{*}}
    \divv_{y}\bmA(\Omega_{\theta})\cdot \divv_{y}\overline{\bmPsi(\Omega_{\theta})} \d y
  $
is
\begin{align}\label{eq:materialDifPterm}
  & \int_{\Omega}\dfrac{1}{\mu_{*}}\divv \bmB(\bmtheta)\divv\overline{\bmvarPsi} \d x
    + \int_{\Omega_{d}}\dfrac{1}{\mu_{*}}
    \divv\bigg((\divv\bmtheta I - \grad\bmtheta - (\grad\bmtheta)^{t}) \bmA\bigg)
    \divv\overline{\bmvarPsi}\d x \notag\\
  & =
    \int_{\Omega}\dfrac{1}{\mu_{*}}\divv \bmB(\bmtheta)\divv\overline{\bmvarPsi} \d x
    - \int_{\Omega_{d}}\dfrac{1}{\mu_{*}}
    \divv((\bmtheta\cdot\grad)\bmA+(\grad\bmtheta)^{t}\bmA)
    \divv\overline{\bmvarPsi}\d x.
\end{align}
We easily get from \eqref{eq:materialDifAterm} and \eqref{eq:materialDifPterm}
the variational formulation \eqref{eq:fvDerivative3dBis} with
$\mathcal{L}(\bmvarPsi,\varPhi)$ given by \eqref{eq:RHSfvDerivative3dTris}.
\end{proof}

\subsection{Expression of the impedance shape derivative using the adjoint state}
Now we shall give a new expression of the impedance measurements
using the above results and the adjoint state.
We recall the expression of the impedance measurements \eqref{eq:impedance3d}
\begin{align*}
  \triangle Z_{kl}(\Omega_{d})
   & = \dfrac{\mi\omega}{|J|^{2}} \int_{\Omega_{d}} \bigg(   (\dfrac{1}{\mu}-\dfrac{1}{\mu^{0}} )\curl \bmA_{k}\cdot\curl\bmA^{0}_{l}
 \\  &\quad- \dfrac{1}{\mi\omega}\left(\sigma-\sigma^{0}\right)(\mi\omega \bmA_{k}+\grad V_{k})\cdot(\mi\omega \bmA^{0}_{l} + \grad V^{0}_{l}) \bigg) \d x.
\end{align*}

\begin{proposition}\label{prop:impedanceDif3d}
  Let $(\bmA_{k}, V_{k})$ be the solution to the variational formulation \eqref{eq:fvEddy3dPotential} with coefficients $\mu$, $\sigma$,
  and $(\bmA^{0}_{l}, V^{0}_{l})$ the solution to \eqref{eq:fvEddy3dPotential} with coefficients $\mu^{0}$, $\sigma^{0}$
  which do not depend on the deposit domain $\Omega_{d}$. Let $(\bmA'_{k},
  V'_{k})$ be the shape derivatives of $(\bmA,V)$.
  Under the same assumptions as in Proposition \ref{prop:fvMateriel}, the shape derivative of the impedance measurement
  $\triangle Z_{kl}(\Omega_{d})$ is given by
\begin{align}\label{eq:impedanceDerivative3d}
   &\triangle Z_{kl}'(\Omega_{d})(\bmtheta)\\
    &\quad= \dfrac{\mi\omega}{|J|^{2}} \int_{\Omega_{d}} \left( (\dfrac{1}{\mu}-\dfrac{1}{\mu^{0}})\curl \bmA'_{k}\cdot\curl \bmA^{0}_{l}
    - \dfrac{1}{\mi\omega}(\sigma-\sigma^{0})(\mi\omega \bmA'_{k}+\grad V'_{k})\cdot(\mi\omega \bmA^{0}_{l} + \grad V^{0}_{l}) \right) \d x \notag \\
  &\quad+ \dfrac{\mi\omega}{|J|^{2}} \int_{\Gamma}(\bmtheta\cdot\bmn)\bigg( \left[\dfrac{1}{\mu}\right]\curl \bmA_{k}\cdot\curl\bmA^{0}_{l}
    - \dfrac{1}{\mi\omega}[\sigma](\mi\omega \bmA_{k\tau}+\grad_{\tau}V_{k})
    \cdot(\mi\omega \bmA^{0}_{l\tau} + \grad_{\tau}V^{0}_{l}) \bigg)\d s.\notag
\end{align}
\end{proposition}
\begin{proof}
From \eqref{eq:impedance3d} one has
\begin{align*}
  \dfrac{|J|^{2}}{\mi\omega}\triangle Z_{kl}(\Omega_{d})
    = \sfa(\Omega_{d})\Big(\bmA_{k},V_{k};\overline{\bmA^{0}_{l}}, -\overline{V^{0}_{l}}\Big)
    - \sfaZero(\Omega_{d})\Big(\bmA^{0}_{l},V^{0}_{l};\overline{\bmA_{k}}, -\overline{V_{k}}\Big),
\end{align*}
where $\sfa$ and $\sfaZero$ are the forms defined in \eqref{eq:formAlpha3d}
with respectively the coefficients $(\mu, \sigma)$ and $(\mu^{0},\sigma^{0})$.
As $(\bmA_{k},V_{k})$ (resp. $(\bmA^{0}, V^{0}_{k})$) satisfies \eqref{eq:lemmCompute}
with constant coefficients $(\mu$, $\sigma)$ (resp. $(\mu^{0}$, $\sigma^{0})$) in $\Omega_{d}$,
Lemma \ref{Chap5_lemm:compute3d} implies
\begin{align}\label{eq:dZcalcul0}
  &\dfrac{|J|^{2}}{\mi\omega}\triangle Z'_{kl}(\Omega_{d})(\bmtheta)\\
    & = \sfaprime(\Omega_{d})(\bmtheta) \Big(\bmA_{k},V_{k};\overline{\bmA^{0}_{l}}, -\overline{V^{0}_{l}}\Big)
    - \sfaprimeZero(\Omega_{d})(\bmtheta) \Big(\bmA^{0}_{l},V^{0}_{l};\overline{\bmA_{k}}, -\overline{V_{k}}\Big) \notag \\[2mm]
  & = \sfa(\Omega_{d})\Big(\bmA'_{k}(\bmtheta),V'_{k}(\bmtheta);\overline{\bmA^{0}_{l}}, -\overline{V^{0}_{l}}\Big)
    + \sfa(\Omega_{d})\Big(\bmA_{k},V_{k};\overline{\bmB^{0}_{l}(\bmtheta)}, -\overline{U^{0}_{l}(\bmtheta)}\Big) \notag \\[2mm]
  \vspace{30pt}
  &\quad - \sfaZero(\Omega_{d})\Big({\bmA^{0}_{l}}'(\bmtheta),{V^{0}_{l}}'(\bmtheta);\overline{\bmA_{k}}, -\overline{V_{k}}\Big)
    - \sfaZero(\Omega_{d})\Big(\bmA^{0}_{l},V^{0}_{l};\overline{\bmB_{k}(\bmtheta)}, -\overline{V_{k}(\bmtheta)}\Big) \notag \\[1mm]
  &\quad + \int_{\Gamma} \bigg(\dfrac{1}{\mu}(\bmtheta\cdot\curl\bmA_{k})(\bmn\cdot\curl\bmA^{0}_{l})
    - \dfrac{1}{\mu^{0}}(\bmtheta\cdot\curl\bmA^{0}_{l})(\bmn\cdot\curl\bmA_{k}) \bigg) \d s \notag \\
  &\quad - \dfrac{1}{\mi\omega} \int_{\Gamma}[\sigma](\bmtheta\cdot\bmn)
    (\mi\omega\bmA_{k\tau}+\grad_{\tau}V_{k})\cdot(\mi\omega\bmA^{0}_{l\tau}+\grad_{\tau}V^{0}_{l}) \d s,
\end{align}
where $(\bmB_{k}(\bmtheta), U_{k}(\bmtheta))$, $(\bmB^{0}_{l}(\bmtheta), U^{0}_{l}(\bmtheta))$
are the material derivatives of $(\bmA_{k}, V_{k})$ and $(\bmA^{0}_{l}, V^{0}_{l})$ respectively.
Now we will compute term by term \eqref{eq:dZcalcul0}. Remark at first that
\begin{align*}
  \sfaZero(\Omega_{d})\Big({\bmA^{0}_{l}}'(\bmtheta),{V^{0}_{l}}'(\bmtheta);\overline{\bmA_{k}}, -\overline{V_{k}}\Big) = 0
\end{align*}
because the shape derivatives $({\bmA^{0}_{l}}'(\bmtheta),{V^{0}_{l}}'(\bmtheta))$ vanish as the potentials
$(\bmA^{0}_{l}, V^{0}_{l})$ in the deposit-free configuration do not depend on $\Omega_{d}$.
This, together with \eqref{eq:dAdVmaterial}, also implies
\begin{align*}
  \bmB^{0}_{l}(\bmtheta) = (\bmtheta\cdot\grad)\bmA^{0}_{l}+(\grad\bmtheta)^{t}\bmA^{0}_{l}
  \quad \text{and} \quad
  U^{0}_{l}(\bmtheta) = \bmtheta\cdot\grad V^{0}_{l}.
\end{align*}
Hence, by substituting $(\bmB^{0}_{l}(\bmtheta),  U^{0}_{l}(\bmtheta))$ with the above expressions, one gets
\begin{align*}
   \sfa(\Omega_{d})\Big(\bmA_{k},V_{k};\overline{\bmB^{0}_{l}}(\bmtheta), -\overline{U^{0}_{l}(\bmtheta)}\Big)
    &=  \sfa(\Omega_{d})\left(\bmA_{k},V_{k};
    \overline{(\bmtheta\cdot\grad)\bmA^{0}_{l}+(\grad\bmtheta)^{t}\bmA^{0}_{l}}, -\overline{\bmtheta\cdot\grad V^{0}_{l}}\right)\\
    &= \mathcal{S}_{1} - \mathcal{S}_{2} \notag
    \end{align*}   
with
\begin{align*}
  & \mathcal{S}_{1} = \int_{\Omega_{d}} \dfrac{1}{\mu}\curl\bmA_{k}
    \cdot\curl\big((\bmtheta\cdot\grad)\bmA^{0}_{l}+(\grad\bmtheta)^{t}\bmA^{0}_{l}\big) \d x \notag \\
  & \mathcal{S}_{2} = \dfrac{1}{\mi\omega}\int_{\Omega_{d}} \sigma\big(\mi\omega\bmA_{k} + \grad V_{k}\big)
    \cdot\big(\mi\omega((\bmtheta\cdot\grad)+(\grad\bmtheta)^{t})\bmA^{0}_{l} + \grad(\bmtheta\cdot\grad V^{0}_{l})\big) \d x.
\end{align*}
We compute $\mathcal{S}_{1}$ and $\mathcal{S}_{2}$
\begin{align*}
  \mathcal{S}_{1} = & \int_{\Omega_{d}}\dfrac{1}{\mu}\curl\bmA_{k}\cdot\curl(\curl\bmA^{0}_{l}\times \bmtheta) \d x \notag \\
  = & \int_{\Omega_{d}} \curl(\dfrac{1}{\mu}\curl\bmA_{k}) \cdot  (\curl\bmA^{0}_{l}\times \bmtheta) \d x
    + \int_{\Gamma} \dfrac{1}{\mu}(\curl\bmA_{k}\times\bmn)\cdot(\curl\bmA^{0}_{l}\times\bmtheta) \d s \notag \\
  = & \int_{\Omega_{d}} \sigma(\mi\omega\bmA_{k} + \grad V_{k})\cdot(\curl\bmA^{0}_{l}\times \bmtheta) \d x
    + \int_{\Gamma} \dfrac{1}{\mu}(\curl\bmA_{k}\times\bmn)\cdot(\curl\bmA^{0}_{l}\times\bmtheta) \d s,
\end{align*}
and
\begin{align*}
  \mathcal{S}_{2} = & \dfrac{1}{\mi\omega}\int_{\Omega_{d}} \sigma\big(\mi\omega\bmA_{k} + \grad V_{k}\big)
    \cdot\bigg(\mi\omega\big(\grad(\bmtheta\cdot\bmA^{0}_{l})+ \curl\bmA^{0}_{l} \times \bmtheta\big)
     + \grad(\bmtheta\cdot\grad V^{0}_{l})\bigg) \d x \notag \\
  = & \dfrac{1}{\mi\omega}\int_{\Omega_{d}} \sigma \big(\mi\omega\bmA_{k} + \grad V_{k}\big)
    \cdot \bigg( (\mi\omega\curl\bmA^{0}_{l} \times \bmtheta)
    + \grad\big(\bmtheta\cdot(\mi\omega \bmA^{0}_{l}+\grad V_{k})\big)\bigg) \d x \notag \\
  = & \dfrac{1}{\mi\omega}\int_{\Omega_{d}} \sigma \big(\mi\omega\bmA_{k} + \grad V_{k}\big)
    \cdot (\mi\omega\curl\bmA^{0}_{l} \times \bmtheta) \d x.
\end{align*}
The last equality is obtained by integration by parts and by the fact that
$\divv(\sigma(\mi\omega\bmA_{k} + \grad V_{k})) = 0$ in $\Omega_{d}$
and that $\sigma(\mi\omega\bmA_{k} + \grad V_{k})\cdot \bmn = 0$
on $\Gamma$. Therefore
\begin{align}\label{eq:dZcalcul1}
  & \sfa(\Omega_{d})\Big(\bmA_{k},V_{k};\overline{\bmB^{0}_{l}}(\bmtheta), -\overline{U^{0}_{l}(\bmtheta)}\Big)
    = \mathcal{S}_{1} - \mathcal{S}_{2}
    = \int_{\Gamma} \dfrac{1}{\mu}(\curl\bmA_{k}\times\bmn)\cdot(\curl\bmA^{0}_{l}\times\bmtheta) \d s \notag \\
  & = \int_{\Gamma} \dfrac{1}{\mu} \bigg(
    (\bmtheta\cdot\bmn)(\curl\bmA_{k}\cdot\curl\bmA^{0}_{l})
    - (\bmtheta\cdot\curl\bmA_{k})(\bmn\cdot\curl\bmA^{0}_{l})\bigg) \d s.
\end{align}
Similarly, we have
\begin{equation}\label{eq:dZcalcul2}
\begin{array}{lll}
&&   \sfaZero(\Omega_{d})\Big(\bmA^{0}_{l},V^{0}_{l};\overline{\bmB_{k}}(\bmtheta), -\overline{U_{k}(\bmtheta)}\Big)\\
  &&=\sfaZero(\Omega_{d})\Big(\bmA^{0}_{l},V^{0}_{l};\overline{\bmA'_{k}}(\bmtheta), -\overline{V'_{k}(\bmtheta)}\Big) \\
     &&\qquad +\sfaZero(\Omega_{d})\Big(\bmA^{0}_{l},V^{0}_{l}; \overline{(\bmtheta\cdot\grad)\bmA_{k}+(\grad\bmtheta)^{t}\bmA_{k}}, -\overline{\bmtheta\cdot\grad V_{k}}\Big)   \\
  &&=\quad \sfaZero(\Omega_{d})\Big(\bmA'_{k}(\bmtheta),V'_{k}(\bmtheta);\overline{\bmA^{0}_{l}}, -\overline{V^{0}_{l}}\Big)\\
     &&\qquad+ \displaystyle\int_{\Gamma} \dfrac{1}{\mu^{0}} \bigg(   (\bmtheta\cdot\bmn)(\curl\bmA_{k}\cdot\curl\bmA^{0}_{l})  - (\bmn\cdot\curl\bmA_{k})(\bmtheta\cdot\curl\bmA^{0}_{l})\bigg) \d s.
\end{array}\end{equation}
From \eqref{eq:dZcalcul0}, \eqref{eq:dZcalcul1} and \eqref{eq:dZcalcul2},
and considering the fact that the support of $\bmtheta$ is on a vicinity of $\Gamma$,
we get \eqref{eq:impedanceDerivative3d}.
\end{proof}

On $\Gamma$, we have
\begin{align*}
  \curl\bmA_{k} \cdot \curl\bmA^{0}_{l}
  & = (\bmn\cdot\curl\bmA_{k})(\bmn\cdot\curl\bmA^{0}_{l})
    + (\curl\bmA_{k}\times\bmn)\cdot(\curl\bmA^{0}_{l}\times\bmn).
\end{align*}
With the above equality and the relations \eqref{eq:dAdVmaterial}, it follows that
\begin{equation}\label{eq:impedanceDerivative3dBis}
\begin{array}{lll}
&&   \triangle Z_{kl}'(\Omega_{d})(\bmtheta) \notag\\
  =\hspace{-.28in}&& \dfrac{\mi\omega}{|J|^{2}}\displaystyle \int_{\Omega_{d}} \bigg\{ (\dfrac{1}{\mu}-\dfrac{1}{\mu^{0}}) \curl \bmB_{k}\cdot\curl \bmA^{0}_{l}
    -\dfrac{1}{\mi\omega} (\sigma-\sigma^{0})(\mi\omega \bmB_{k}+\grad U_{k})\cdot(\mi\omega \bmA^{0}_{l} + \grad V^{0}_{l}) \bigg\} \d x  \\
  && - \dfrac{\mi\omega}{|J|^{2}} \displaystyle\int_{\Omega_{d}} \bigg\{  (\dfrac{1}{\mu}-\dfrac{1}{\mu^{0}})
    \curl((\bmtheta\cdot\grad)\bmA_{k} + (\grad\bmtheta)^{t}\bmA_{k})\cdot\curl \bmA^{0}_{l}  \\
  &&-\dfrac{1}{\mi\omega}(\sigma-\sigma^{0})\bigg(\mi\omega \big((\bmtheta\cdot\grad)\bmA_{k}+(\grad\bmtheta)^{t}\bmA_{k}\big)+\grad (\bmtheta\cdot\grad U_{k})\bigg)
    \cdot(\mi\omega \bmA^{0}_{l}+\grad V^{0}_{l}) \bigg\} \d x  \\
  && + \dfrac{\mi\omega}{|J|^{2}} \displaystyle\int_{\Gamma}(\bmtheta\cdot\bmn)\bigg\{ \left[\dfrac{1}{\mu}\right]    (\bmn\cdot\curl \bmA_{k})(\bmn\cdot\curl \bmA^{0}_{l})\\
   && -  \left[\dfrac{1}{\mu}\right](\dfrac{1}{\mu}\curl \bmA_{k} \times \bmn)\cdot (\dfrac{1}{\mu^{0}}\curl \bmA^{0}_{l} \times \bmn)  \\
  &&-\dfrac{1}{\mi\omega}[\sigma](\mi\omega \bmA_{k\tau}+\grad_{\tau} V_{k})    \cdot(\mi\omega \bmA^{0}_{l\tau} + \grad_{\tau} V^{0}_{l}) \bigg\}\d s.
    \end{array}
\end{equation}
	We follow the method of Hadamard representation to give an expression of
$Z_{kl}'(\Omega_{d})(\bmtheta)$ dependent of
($(\bmA'(\bmtheta),V'(\bmtheta))$ or $(\bmB(\bmtheta),U(\bmtheta))$)
of the solution $(\bmA,V)$ by introducing the adjoint state
$(\bmP_{l},W_{l})\in \X$ related to the solution
$(\bmA^{0}_{l},V^{0}_{l})$ in the deposit-free case. The adjoint problem writes
\begin{align}\label{eq:fvAdjoint3d}
  \mathcal{S}^{*}(\bmP_{l},W_{l};\bmPsi,\Phi) = L^{*}(\bmPsi,\Phi) \qquad
  \forall (\bmPsi,\Phi) \in \X,
\end{align}
where for any $(\bmA,V)$, $(\bmPsi,\Phi)$ in $\X$ we have $\mathcal{S}^{*}(\bmA,V;\bmPsi,\Phi):= \overline{\mathcal{S}(\bmPsi,\Phi;\bmA,V)}$ and
\begin{align*}
  L^{*}(\bmPsi,\Phi) :=& \int_{\Omega_{d}} (\dfrac{1}{\mu}-\dfrac{1}{\mu^{0}})\curl\overline{\bmA^{0}_{l}}\cdot \curl\overline{\bmPsi}
\\ &+\int_{\Omega_{d}} \dfrac{1}{\mi\omega}(\sigma-\sigma^{0})(\overline{\mi\omega \bmA^{0}_{l} + \grad V^{0}_{l}})
    \cdot(\overline{\mi\omega\bmPsi + \grad \Phi}) \d x.
\end{align*}
From the above considerations we easily derive the jumps condition for the adjoint states $(\bmP_l)$ 
\begin{equation}\label{eq:PotentialAdj}
\begin{cases}
      \left[\bmn\cdot \curl \bmP_{l}\right]=0 \quad &\text{on } \Gamma,\\
      \left[\mu^{-1}\curl \bmP_{l}\times \bmn\right] = -(\dfrac{1}{\mu}-\dfrac{1}{\mu^{0}})\curl\overline{\bmA^{0}_{l}}\times \bmn  \quad &\text{on } \Gamma.\\
      \end{cases}
      \end{equation}
It is worth noticing that the adjoint state $\bmP_l$ satisfies the Coulomb gauge condition.
We are now in a position to express the results of
proposition~\ref{prop:impedanceDif3d} with the use of the adjoints states
$(\bmP_l,W_l)$. Indeed, we have the following proposition (which is an
immediate consequence of Proposition \ref{prop:impedanceDif3d} and the definition of the
adjoint state)
\begin{proposition}\label{prop:impedanceDerivative3d}
  Let $(\bmA_{k}, V_{k})$ be the potentials induced by the coil $k$
  of the eddy-current problem with deposit domain $\Omega_{d}$, $(\bmA^{0}_{l}, V^{0}_{k})$
  the potentials induced by the coil $l$ for the deposit free case,
  and $(\bmP_{l}, W_{l})$ the adjoint states related to $(\bmA^{0}_{l}, V^{0}_{k})$
  which satify the adjoint problem \eqref{eq:fvAdjoint3d}.
  Then under the same assumptions as in Theorem \ref{theo:shapeContinuity} for $\mu$ and $\sigma$,
  the impedance shape derivative \eqref{eq:impedanceDerivative3dBis} can be
  expressed as
  \begin{align}\label{eq:impedanceDerivative3dFinal}
    \triangle Z_{kl}'(\Omega_{d})(\bmtheta)
    & =  \dfrac{\mi\omega}{|J|^{2}}\int_{\Gamma} (n\cdot\bmtheta)\bigg\{
      \left[\dfrac{1}{\mu}\right](\bmn\cdot\curl \bmA_{k})(\bmn\cdot\overline{\bmP_{l}}-\bmn\cdot\curl \bmA^{0}_{l}) \notag \\
    & \qquad \qquad  - [\mu] \left(\dfrac{1}{\mu}\curl \bmA_{k}\times \bmn\right)
      \cdot \left(\dfrac{1}{\mu^{0}}(\curl\overline{\bmP_{l}})_{+}\times \bmn
      - \dfrac{1}{\mu^{0}}\curl \bmA^{0}_{l}\times \bmn\right) \notag \\
    & \qquad \qquad  + \dfrac{1}{\mi\omega}[\sigma](\mi\omega \bmA_{k\tau}+\grad_{\tau}V_{k})
      \cdot(\overline{\mi\omega \bmP_{l\tau}+\grad_{\tau}W_{l}} + \mi\omega \bmA^{0}_{l\tau}+\grad_{\tau}V^{0}_{l})
      \bigg\} \d s.
  \end{align}
\end{proposition}

\subsection{Explicit shape gradient formula}
The shape derivative of $\mathbf{f}(\Omega_{d})$ is in the form
\begin{align}
  \mathbf{f}'(\Omega_{d})(\bmtheta) = -\dfrac{\omega}{|J|^{2}} \int_{\Gamma_{0}}(\bmn\cdot\bmtheta)g \d s,
\end{align}
where the shape-dependent function $g$ depends on the solutions to the forward problem
$(\bmA_{k}, V_{k})$, $(\bmA^{0}_{l}, V^{0}_{l})$ and the adjoint state
$(\bmP_{l}, W_{l})$. More precisely, $g=g_{11}+g_{21}$ for the absolute mode, and $g=g_{11}-g_{22}$ for the differential mode. For any $l$ and $k$ we have
\begin{align}\label{eq:gfunction3d}
  g_{kl} = \int_{\zeta_{\min}}^{\zeta_{\max}} & \Re \bigg( (\overline{Z(\Omega_{d};\zeta) - Z^{\natural}(\zeta)}) \bigg\{
    \left[\dfrac{1}{\mu}\right](\bmn\cdot\curl \bmA_{k})(\bmn\cdot\curl\overline{\bmP_{l}}-\bmn\cdot\curl \bmA^{0}_{l}) \notag \\
  & - [\mu] \left(\dfrac{1}{\mu}\curl \bmA_{k}\times \bmn\right)
    \cdot \left(\dfrac{1}{\mu^{0}}\curl\overline{\bmP_{l}}\times \bmn
    - \dfrac{1}{\mu^{0}}\curl \bmA^{0}_{l}\times \bmn\right) \notag \\
  & + \dfrac{1}{\mi\omega}[\sigma](\mi\omega \bmA_{k\tau}+\grad_{\tau}V_{k})
    \cdot(\overline{\mi\omega \bmP_{l\tau}+\grad_{\tau}W_{l}} + \mi\omega \bmA^{0}_{l\tau}+\grad_{\tau}V^{0}_{l})
    \bigg\}\bigg) \d \zeta.
\end{align}
We choose the shape perturbation $\bmtheta$ such that $ \bmtheta = g \bmn$  on the interface $\Gamma$, which is a descent direction since
\begin{align*}
  \mathbf{f}'(\Omega_{d})(\bmtheta) = -\dfrac{\omega}{|J|^{2}} \int_{\Gamma_{0}} |g|^{2} \d s \leq 0.
\end{align*}

\section{Numerical algorithms for the deposit reconstruction}\label{sec:num}
We recall that the computational domain $\Omega$ is a cylinder that contains
the tube and the SP. We introduce a family of triangulation $\mathcal{T}_{h}$ of $\Omega$, the subscript $h$ stands for the largest length of the edges in $\mathcal T_{h}$. The tetrahedrons of $\mathcal T_{h}$ match on the interface between the conductive part (i.e. tube and SP $\sigma\neq 0$) and the insulator part ($\sigma=0$). The triangulation of the conductive parts (with deposits region) is given in Fig.~\ref{imageandmesh} (see (A)-(B) for a real image and (A')-(B') for its F.E model).
\begin{figure}[htbp]
\centering
\begin{tabular}{cc}
\includegraphics[width=4cm,height=3cm]{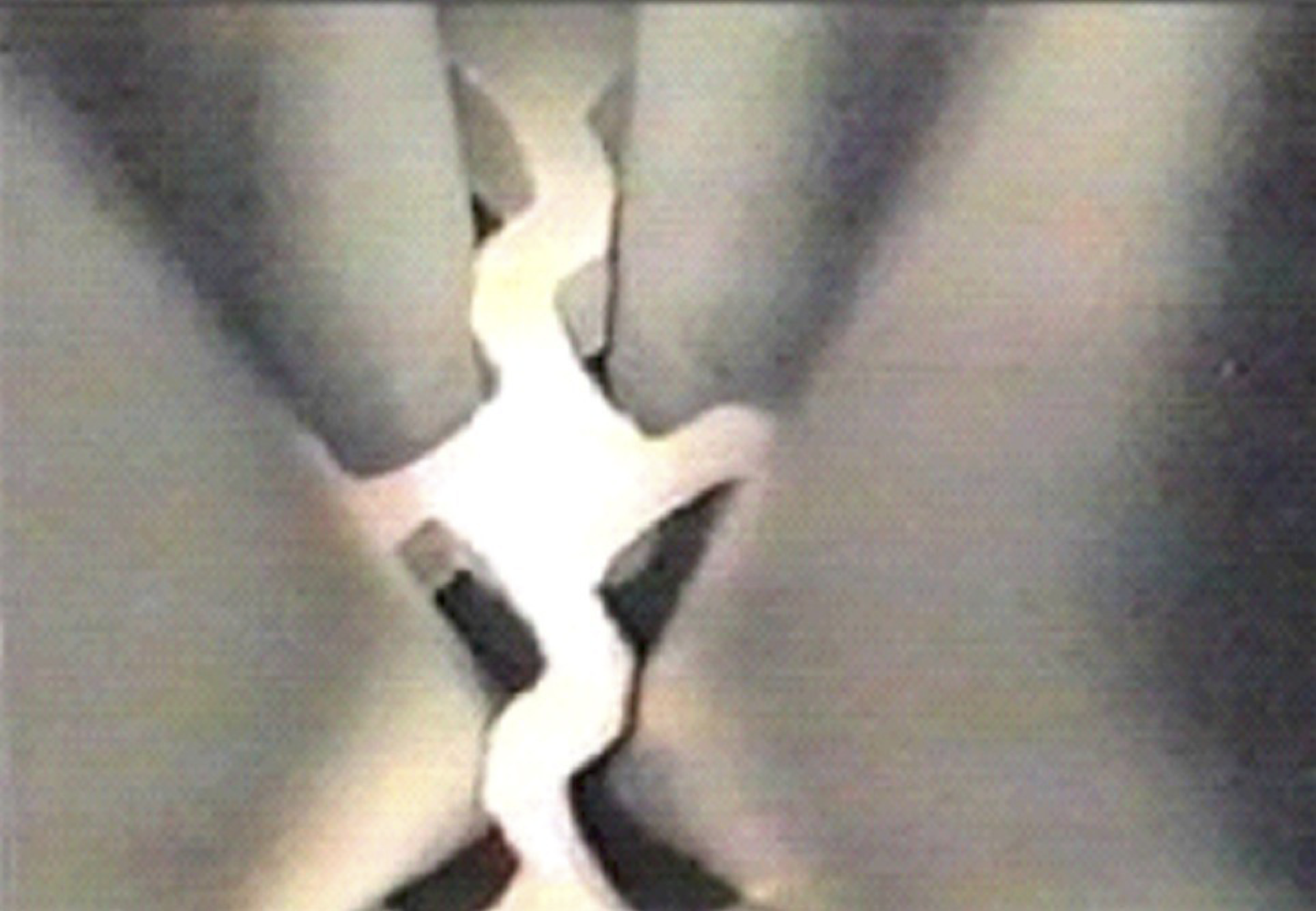}&
\includegraphics[width=4cm,height=3cm]{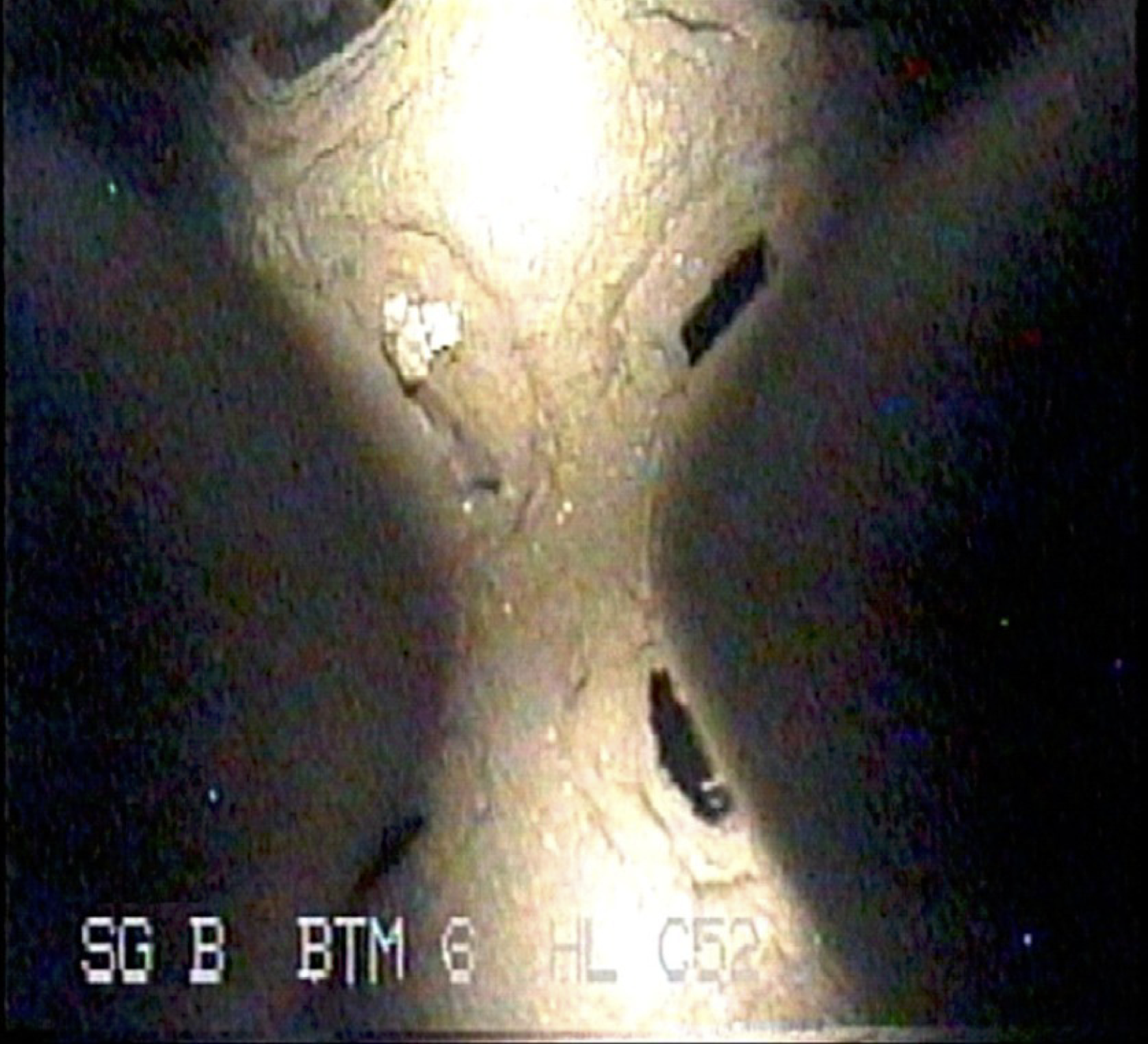}\\
(A)&(B)\vspace{.2in}\\
\includegraphics[width=4cm,height=3cm]{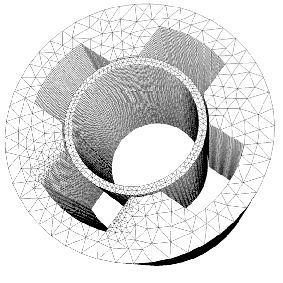}&
\includegraphics[width=4cm,height=3cm]{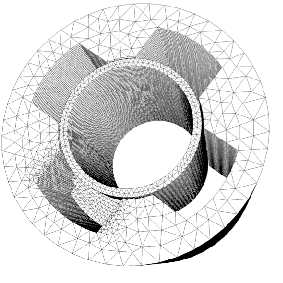}\\
(A')&(B')\\
\end{tabular}
\caption{Pictures downloaded from $\copyright$-Westinghouse:
  http://westinghousenuclear.com" . Figure (A) presents non clogging (healthy)
  SP, while Figure (B) presents a fully and partially clogged SP. Figures
  (A')-(B') are representations of used meshes for each configuration respectively.}
\label{imageandmesh}
\end{figure}
Since the variational space (of the regularized variational formulation) is
based on $H^1$ functions, the numerical finite elements approximation will be
based on nodal finite elements for the electric vector potential $V$ as well as
for the magnetic vector potential $\bmA$. We shall mainly use $\mathbb P_1$
Lagrange nodal elements for both.
In addition, the boundary conditions ($\bmA\cdot\bmn =0$ on $\partial\Omega$)
are taken into account via penalization of degrees of freedoms that belong to $\partial\Omega$. 
Indeed the same numerical approximation procedure is applied to adjoint states.

We now describe the gradient descent algorithm steps, the geometrical
parametrizations and the procedure to accelerate iterations steps. The deposit is
assumed to be located on the outer part of the tube and is concentrated (for
the non axisymmetric examples) in one opening part of the SP (see
Fig.~\ref{imageandmesh}-(B)). The reconstruction is based on an intuitive
approach, which consists in iteratively P0-approximating the geometry of the
deposit on a predefined 3D grid. This method avoid to reconstruct the mesh at
each inversion iteration. The predefined grid is defined by $\mathcal{N}_{\underline{h}}=\{\mathcal T_h^d\subset\mathcal T_h, \text{ s.t } \forall \text{ simplex } K\in \mathcal T_h^d\   \text{ with a facet parallel to Tube} \}$, where $\underline h$ stands for the resolution of the grid. We give in Fig.~\ref{GridClipping} a clipping of $\mathcal{N}_{\underline{h}}$.
 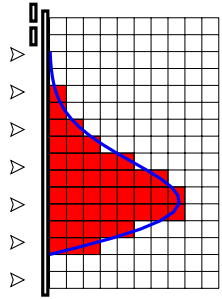
\begin{figure}[htbp]
 \centering
\begin{tikzpicture}[scale=0.45]
\begin{scope}
\draw [fill=red] (0,0) rectangle (1.5,.5);%
\draw [fill=red] (0,.5) rectangle (2.5,1);%
\draw [fill=red] (0,1) rectangle (4,1.5);%
\draw [fill=red] (0,1.5) rectangle (4,2);%
\draw [fill=red] (0,2) rectangle (3.5,2.5);%
\draw [fill=red] (0,2.5) rectangle (2.5,3);
\draw [fill=red] (0,3) rectangle (1.5,3.5);
\draw [fill=red] (0,3.5) rectangle (1,4);
\draw [fill=red] (0,4) rectangle (.5,4.5);
\draw [fill=red] (0,4.5) rectangle (.5,5);
\draw[thick,very thin] (0,-1) grid [step=.5cm] (5,7);
\end{scope}
\begin{scope}[domain=-6:0]
\draw[very thick,color=blue,rotate=-90] plot (\x,{-4*\x*exp(-.2*\x*\x)}) ;
\end{scope}
\begin{scope}[decoration={shape backgrounds,shape=dart,shape size=5pt, shape sep=5mm,shape sloped=false}]
\draw [black] decorate {(-1,-.75) -- (-1,6.75)};
\end{scope}
\begin{scope}
\draw[very thick] (-.2,-1.2) rectangle (-.05,7.2);
\end{scope}
\begin{scope}
\draw[very thick] (-.55,6.2) rectangle (-.4,6.7);\draw[very thick] (-.55,6.9) rectangle (-.4,7.4);
\end{scope}
\end{tikzpicture}
\hspace{1in}
\caption{Sketch of the deposit shape reconstruction (clipping of 3D representation)
  on a grid $G_{\underline{h}}$. The reconstruction uses an invariant grid where P0-interpolation over this grid is set to cary out the new shape profile.
}

\label{GridClipping}
\end{figure}
We present in {\bf Algorithm~\ref{algo1}} the instances of an adapted
step gradient descent. It is well known that the fixed step gradient descent
algorithm converges if the step is sufficiently small. In our case we
will allow the step descent to be large at least for the first iterations, and
if the algorithm fails to maintain the decreasing of the cost functional, the
step is reduced by a given factor. 
  $1/2<\delta<1$. The final geometry is the one for which no local variation
  (on the predefined grid) decreases the cost functional. 
\begin{algorithm}\tiny
\LinesNumbered
\DontPrintSemicolon
\SetAlgoLined
\KwData{The impedance signal response of the tested configuration $\Omega_d^\star$}
\KwResult{Optimal shape approximation using interpolation on 3D grid}
\KwIn{ The resolution of the predefined grid $\underline{h}$, ,Threshold: $\epsilon$\;}
\KwIn{ $L^0$ Table of size $N$, a real $t$ and $\dfrac 1 2 < \delta <1$\;}
\KwIn{$P=\mathrm{E}\big(log(0.4*\underline{h})\slash log(\delta)\big)$\;}
{\bf Build the 3d mesh grid $\mathcal{N}_{\underline{h}}: m_x\underline{h}\times m_y\underline{h}\times N$\;}
	Evaluate the cost function $\mathbf{f}(L^{0})$ and the gradient $\nabla\mathbf{f}(L^{0})$\;		
	$k=0$\;
\While{$\|\nabla \mathbf{f}(L^{k})\|_2>\epsilon$}
	{
	$t^{k}=\underline{h} \slash \max_{1\leq n\leq N}|\nabla\mathbf{f}(L_{n}^{k})|$\;
	\For{$1\leq p \leq P$}
		{
	   	$t^{p}=\delta\times t^{p-1}$\;
		$L^{p} = L^{k} -t^p\nabla\mathbf{f}(L^{k})$\;
                Project the $L^{p}$ (to the nearest value) on the predefined grid;
		Evaluate the cost function $\mathbf{f}(L^{p})$ \;		
		\If{$\mathbf{f}(L^{p}) < \mathbf{f}(L^{k})$}
		{
			Update $\mathbf{f} (L^{k+1}_{})=\mathbf{f} (L^p)$\;
			Evaluate the gradient $\nabla\mathbf{f}(L^{k+1})$\;
			
			{\bf Break}\;
		}
		\If{p==P}
		{
			Print ``A singular point is attained"\;
	                {\bf Exit}()\;
		}
		}

		$k=k+1$\;			
	}
\caption{Gradient descent algorithm}\label{algo1}
\end{algorithm}

\section{Numerical implementation and validation}\label{sec:exp}

	Numerical validation of the presented method is considered in this
        section. We use the software FreeFem++ ~\cite{MR3043640} to
        deal with the finite elements discretization of the problem. We run our script on a cluster with
        distributed memory configuration. We use a direct matrix-inversion of the linear system where the factorization is achieved using sparse parallel solver (MUMPS~\cite{MUMPS:1,MUMPS:2}).
	We present and explain in the sequel some particular techniques to
        achieve performance of the direct eddy-current solver (and consequently the inverse solver).
	
	 At each probe position we have to compute a solution associated  to
         different source term.  In order to (numerically) ensure
         divergence free condition for the source term one has to exactly mesh
         the support of the coil. If we build a new mesh related to the new probe
         position, we have to assemble new matrices and solve new systems,
         which are extremely memory-consuming. We therefore avoid this by
         creating and use a unique mesh that incorporates all possible probe
         positions in a scan of the tube. This allows us to only modify the
         right hand side of the system at each coil position. The factorization
         of the matrix is done only once per iteration. In order to further
         accelerate the resolution we also parallelize the matrix assembly
         since the cost of this part appeared to be the more expensive part if not
         done in parallel. Particular attention must be taken for the non-homogeneity (change of the conductivities and the permeability in the domain):  We declare the variables $\sigma$ and $\mu$ as P0-Lagrange finite elements that depends on the elements labels of the non-partitioned mesh. Then, we apply a graph partitioning (e.g. scotch~\cite{Pellegrini01scotchand} or metis~\cite{Karypis95metis}) to create automatically partitioned new mesh. Since the partitioning process changes the elements labels to the ranks of the used group of processors, we define the P0-Lagrange non-homogeneous domain variable on the non-partitioned mesh and then include them in the variational formulation that admits the partitioning (see ~\cite{haddar:hal-01044648} for more technical details).
				
		Numerical experiments deal with several configurations of test cases. Mainly we present an axisymmetric configuration, then we add the SP and consider the case where one of the SP foils (flow path) is clogged.
		
		The geometry of the computational domain includes a tube with respective internal and external radius $9.84$ mm and $11.11$ mm. The coils are modeled by a crown with respective internal and external radius $7.83$ mm and $8.50$ mm. Both coils have length $2$ mm and are separated by $0.5$ mm. The scan step of coils is fixed to $1$ mm and cover $20$ positions along the tube, which length has been limited to $30$ mm.
		
		We used the following values of the electromagnetic parameters. The frequency $\omega=200 \pi$, the magnetic permeability of the vacuum $\mu_0=\pi\times 10^{-6}$, magnetic permeability of the tube $\mu_t=1.01\mu_0$, the magnetic permeability of the SP $\mu_{SP}=\mu_0$ and the magnetic permeability of the deposit $\mu_d=\mu_0$. The conductivity is taken $\sigma_t=1\times 10^{3}$ for the tube, $\sigma_{SP}=1\times 10^{2}$ for the SP and $\sigma_d=\sigma_t$ for the deposit.
					
		In all numerical experiments, the initialization of our algorithm takes a deposit with the lowest layers in the grid $G_{\underline{h}}$ i.e. with depth $0.5$ mm equal to ${\underline{h}}$ : the precision of the fixed grid.
		

\subsection{Axisymmetric and non-axisymmetric geometries}
In this part we consider two configurations of deposits in the vicinity of the
tube: deposits around the tube far from SP and a deposit in one opening of
water traffic lane of the SP. The first case, represents an axisymmetric
configuration~\cite{jiang:hal-00741616} and the second case represents a
non-axisymmetric configuration because of the presence of SP and the deposit.
We present in Figure~\ref{GeoAxis} a slice on the plane (x,z) of the 3D
computational domain. We show the shape of the axisymmetric deposits
$\Omega_d^\star$ and the estimated deposits $\Omega_d^k$ result of the
inversion algorithm. Together with this plot we add the y-component of the solution $\bmE_k$ to show the penetration of the electromagnetic wave inside the tube and the deposits. With respect to $k$, a series of measured responses of the estimated deposit $\Omega_d^k$ is presented in Figure~\ref{SeriesResponseAxis}. This shows the convergence of the method in the sense of minimizing the misfit function~\eqref{eq:leastSquare} presented in Figure~\ref{Jaxi}.
\begin{figure}[pt]
\centering
  \begin{tabular}{cc}
\centering  \includegraphics[width=6cm,height=5.4cm]{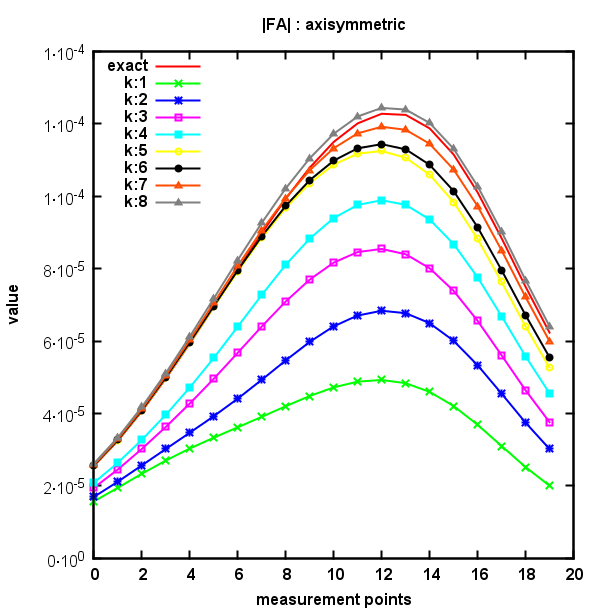},
\centering  \includegraphics[width=6cm,height=5.4cm] {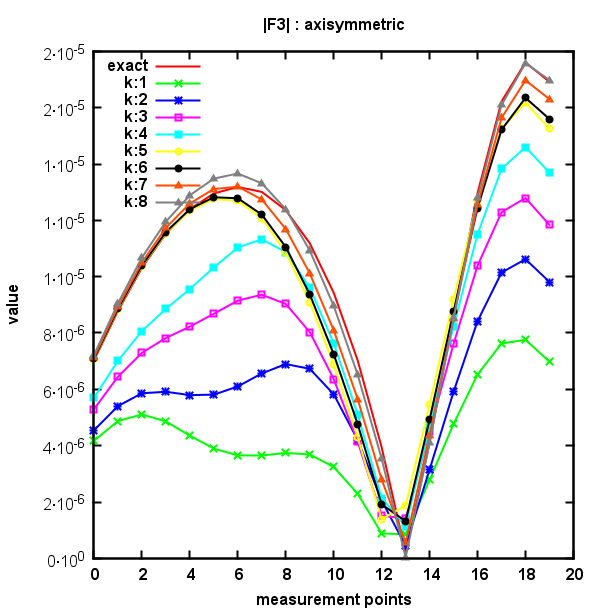}
  \end{tabular}
  \caption{History of the impedance responses of the deposits during iterations
    in the axisymmetric configuration. $|FA|$ measurements (left) and $|F3|$ measurements (right).}
  \label{SeriesResponseAxis}
\end{figure}
\begin{figure}[htpb]
  \begin{tabular}{cc}
\includegraphics[scale=.3]{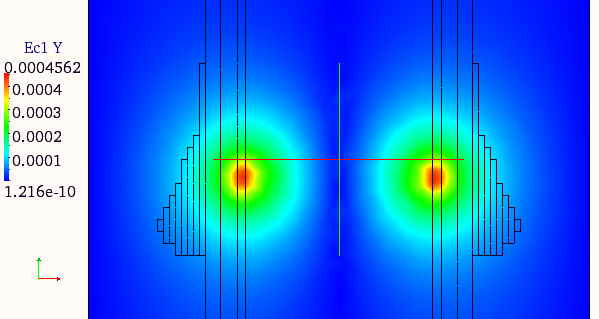}&
\includegraphics[scale=.3]{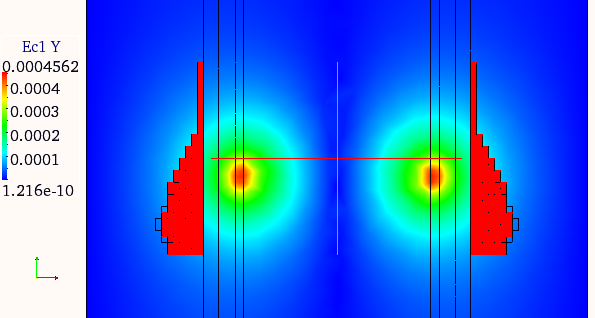}
    \end{tabular}
  \caption{Slice on x-z-plan of the computational domain showing the plot of the y-component of the eddy-current solution and the deposits $\Omega_d^\star$ (on the left) and the reconstruction in red (on the right).}
  \label{GeoAxis}
\end{figure}

\begin{figure}[htbp]
\centering
\includegraphics[width=6cm,height=5.4cm] {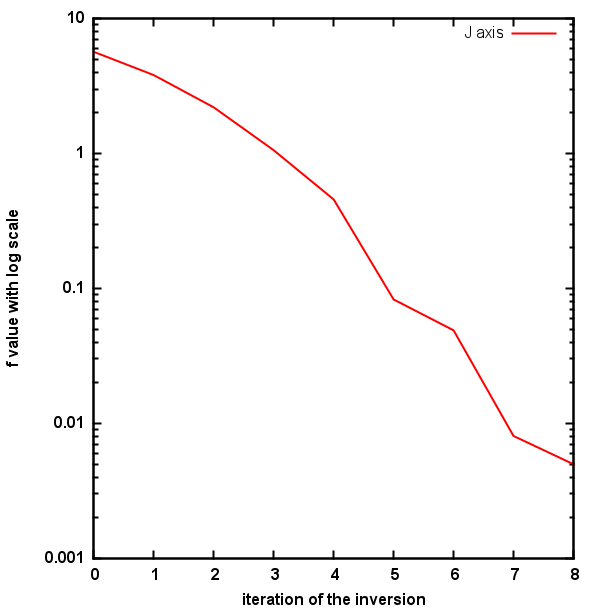}
\caption{Objective function with respect to iterations in the axisymmetric case.}
\label{Jaxi}
\end{figure}
A more complex configuration consists in taking into account the presence of SP
and therefore non symmetric deposit. The results for this configuration are
presented as follows: In Figure~\ref{GeoSP1}  we plot a slice, on the plane (x,z), of the y-component of the solution $\bmE_k$ together with the shape profile of the deposits $\Omega_d^\star$ and its estimation $\Omega_d^k$ result of the inversion algorithm.
\begin{figure}[htp]
\centering
  \begin{tabular}{cc}
\centering  \includegraphics[width=6cm,height=5.4cm]{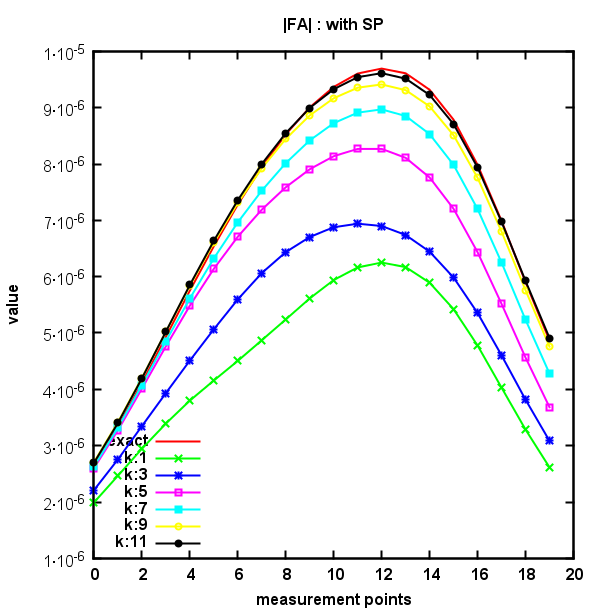},
\centering  \includegraphics[width=6cm,height=5.4cm] {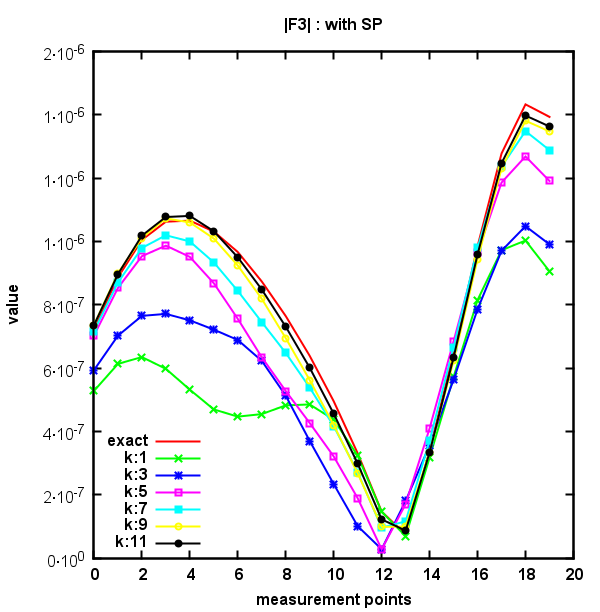}
  \end{tabular}
   \caption{History of the impedances during iterations for the SP configuration. $|FA|$ measurements (left) and $|F3|$ measurements (right).}
  \label{SeriesResponseSP}
  \end{figure}
 The series of the impedance signal responses are given with respect to $k$ in Figure~\ref{SeriesResponseSP}. This highlights the convergence of our algorithm even with the presence of noise in the non symmetric solution (y-component of $\bmE_k$) as it can be seen in Figure~\ref{GeoSP1} and also on the left plot of Figure~\ref{GeoSP2}.
\begin{figure}[htbp]
\centering
  \begin{tabular}{cc}
\includegraphics[scale=.35]{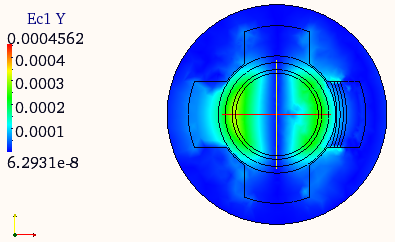}&
\centering\includegraphics[width=6cm,height=5.4cm]{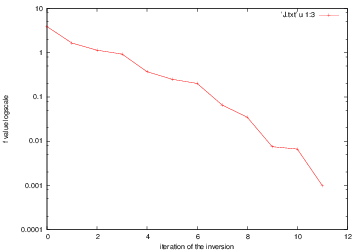}
    \end{tabular}
    \caption{(left) Slice representation of the computational domain of the
      configuration with SP and a plot of the z-component of the eddy-current
      solution on x-y-plan. (right) Misfit function in terms of the inversion iterations.}
\label{GeoSP1}
    \end{figure}

\begin{figure}[htbp]
 \begin{tabular}{cc}
 \includegraphics[scale=.4]{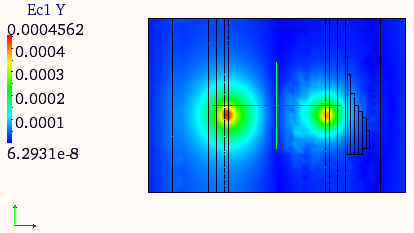}&
\centering\includegraphics[scale=.4]{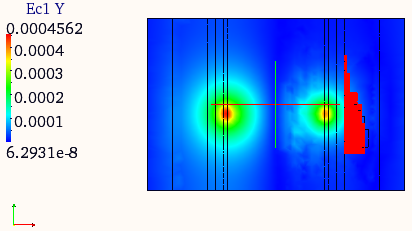}
    \end{tabular}
  \caption{Slice representation of the computational domain of the configuration with SP and a plot of the y-component of the eddy-current solution on x-z-plan.}
    \label{GeoSP2}
\end{figure}

\subsection{Arbitrary deposit shape}
In this subsection in addition to the presence of the SP, we consider the
reconstruction of a deposits with an arbitrary shape that does not match the
parametrization used for the inverse problem: see Figure~\ref{ShapeNoise}. The
results of the inversion algorithm is given (in terms of $k$) in Figure~\ref{GeoSPShape}. The convergence in the sense of the impedance response measurements is given in Figure~\ref{SeriesSPShape}. The minimization of the objective function with respect to the iterations of the inversion is presented in Figure~\ref{JSPNoiseShape}.

\begin{figure}[htp]
\centering
\begin{tabular}{lr}
\includegraphics[width=6cm,height=5.2cm]{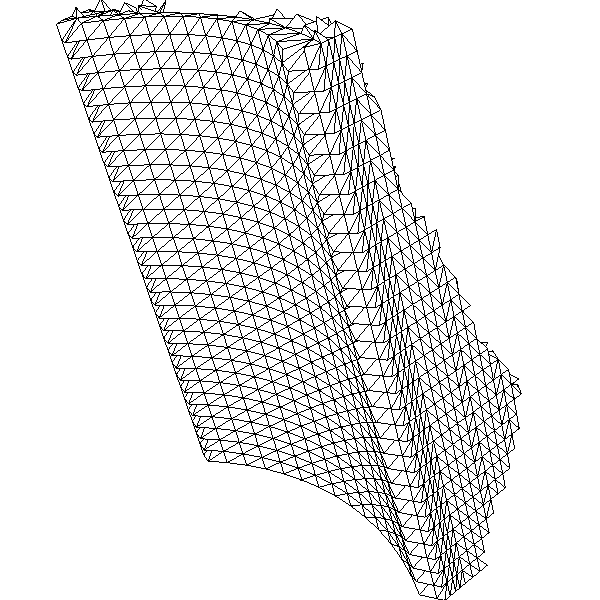}&
\includegraphics[width=6cm,height=5.2cm]{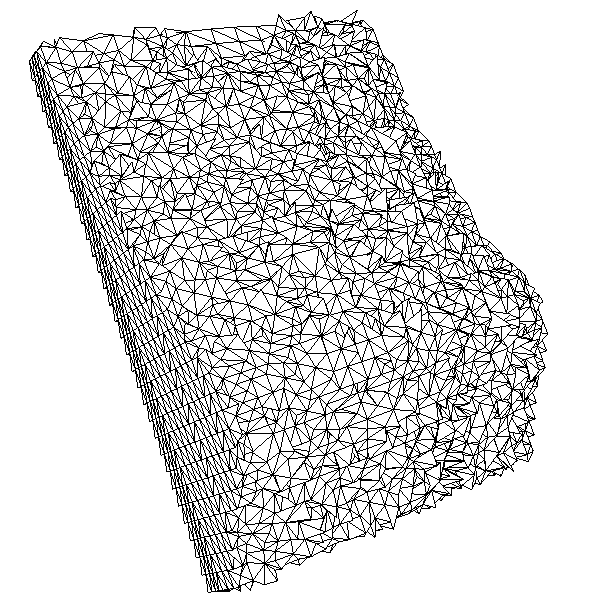}
    \end{tabular}
    \caption{Graph of the arbitrary shaped deposit that clogs one opening (foils) of the tube SP.}
    \label{ShapeNoise}
\end{figure}

\begin{figure}[htp]
  \begin{tabular}{cc}
\centering  \includegraphics[width=6cm,height=5.4cm]
{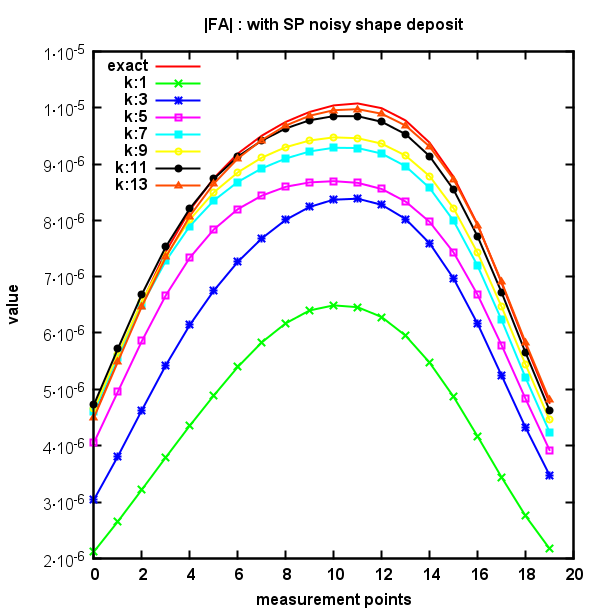},
\centering  \includegraphics[width=6cm,height=5.4cm] {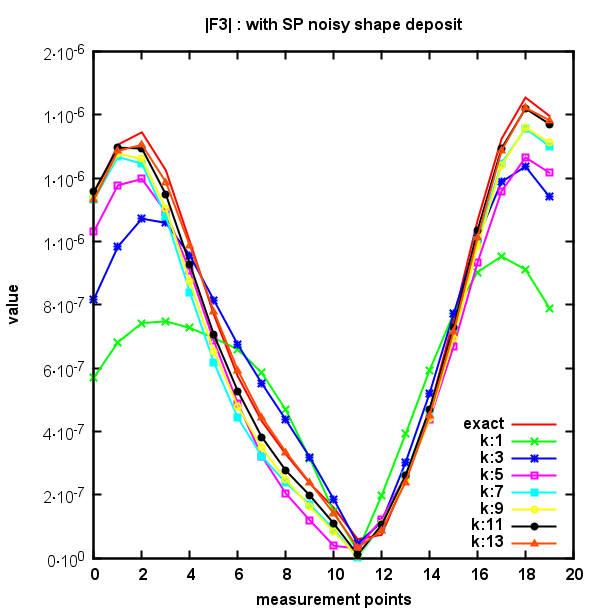}
  \end{tabular}
    \caption{History of the impedances in the case or deposit with arbitrary
      shape:  $|FA|$ measurement (left) and $|F3|$ measurement (right).}
    \label{SeriesSPShape}
\end{figure}

\begin{figure}[htbp]
\centering\includegraphics[width=6cm,height=5.4cm] {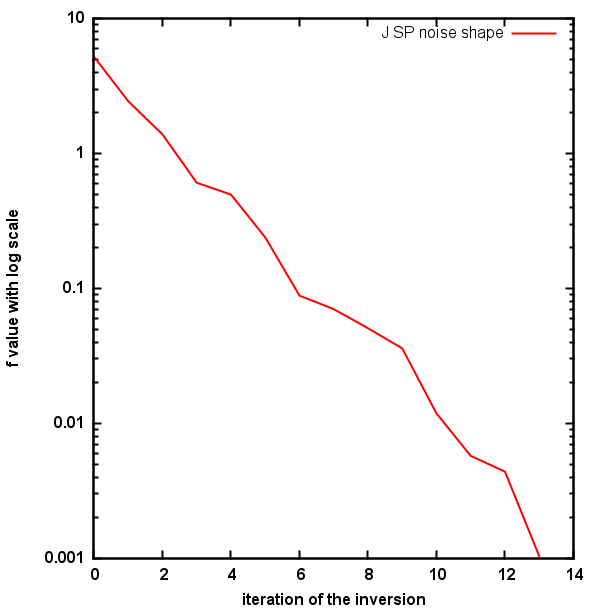}
\caption{Objective function with respect to iterations for the case of
  arbitrary shape reconstructions.}
\label{JSPNoiseShape}
\end{figure}

\begin{figure}[htp]
\begin{tabular}{ccc}
\includegraphics[scale=.2]{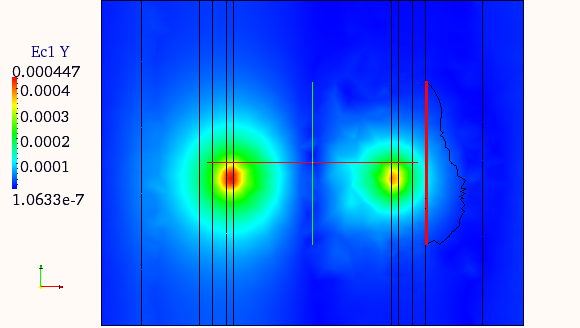}&\includegraphics[scale=.2]{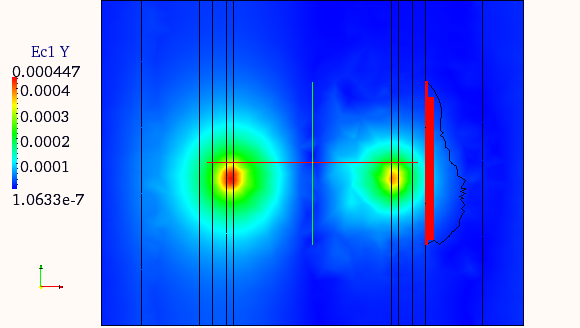}&\includegraphics[scale=.2]{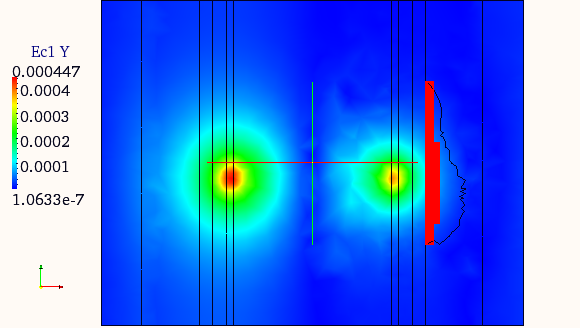}\\(k=0)&(k=2)&(k=4)\\
\includegraphics[scale=.2]{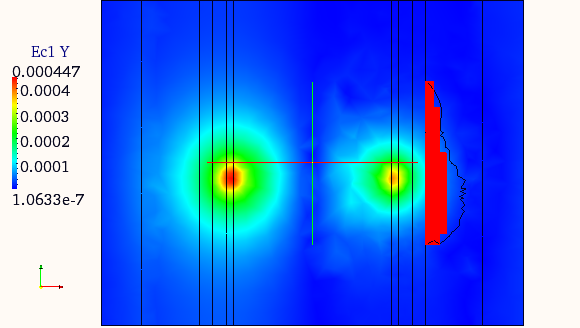}&\includegraphics[scale=.2]{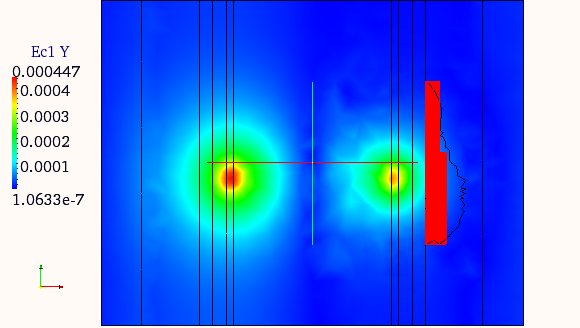}&\includegraphics[scale=.2]{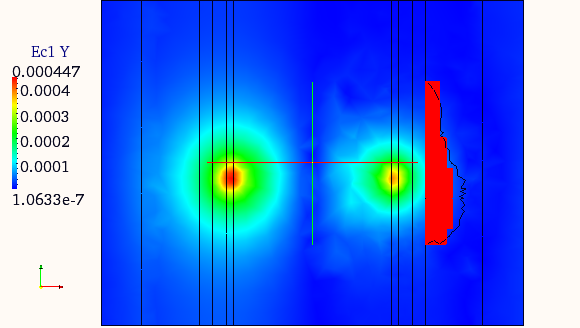}\\(k=6)&(k=8)&(k=10)\\
\includegraphics[scale=.2]{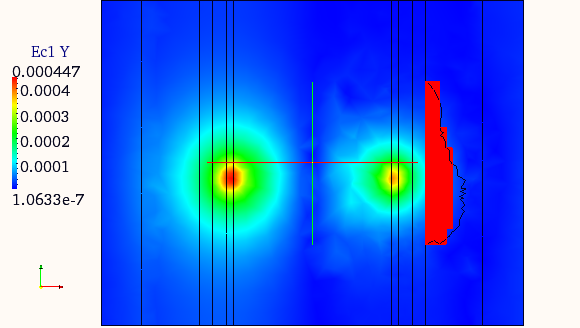}&\includegraphics[scale=.2]{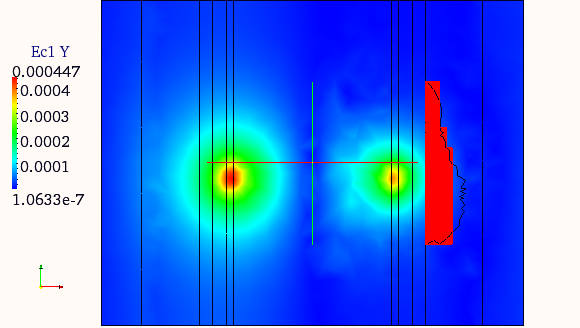}&\includegraphics[scale=.2]{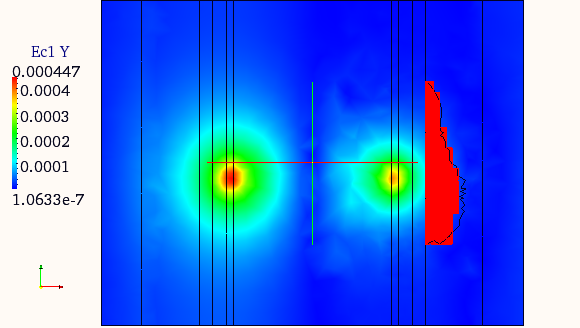}\\(k=11)&(k=12)&(k=13)
\end{tabular}
\caption{Slice representation of the iterations of the shape reconstruction:
  case of an arbitrary shape clogging one opening of the quatrefoil SP.}
\label{GeoSPShape}
\end{figure}




\Appendix
\section{Some useful differential identities}
\begin{subequations}\label{eq:VectorFormula}
  \begin{align}
    & \curl (\grad f) = 0, \label{eq:VectorFormulaCurl0}\\
    & \divv (\curl \bm{v}) = 0, \label{eq:VectorFormulaDivv0}\\
    & (\bm{u}\cdot\grad)\bm{v} = (\grad \bm{v}) \bm{u}, \label{eq:VectorFormula1}\\
    & \curl \bm{u} \times \bm{v} = (\grad \bm{u} - (\grad \bm{u})^{t}) \bm{v}, \label{eq:VectorFormula2}\\
    & \grad(\bm{u}\cdot \bm{v}) = \bm{u}\times\curl \bm{v} + \bm{v}\times\curl \bm{u}
      +(\bm{u}\cdot\grad)\bm{v}+(\bm{v}\cdot\grad)\bm{u}, \label{eq:VectorFormula3}\\
    & \curl(\bm{u}\times \bm{v}) = \bm{u}\divv \bm{v} - \bm{v}\divv \bm{u}
      +(\bm{v}\cdot\grad) \bm{u} - (\bm{u}\cdot\grad) \bm{v}. \label{eq:VectorFormula4}
  \end{align}
\end{subequations}
\appendix
\section{Proof of Lemma~\ref{Chap5_lemm:compute3d}}
We develop the proof of the shape derivative calculus presented at Lemma \ref{Chap5_lemm:compute3d}
\begin{proof}
By definition, one has
\begin{align*}
  \sfa(\Omega_{\theta})\big(\bmA,V;\bmPsi,\Phi\big)
    = & \int_{\Omega_{\theta}}\dfrac{1}{\mu}\curl_{y} \bmA(\Omega_{\theta}) \cdot \curl_{y} \overline{\bmPsi(\Omega_{\theta})} \d y \notag \\
  & + \dfrac{1}{\mi\omega}\int_{\Omega_{\theta}} \sigma(\mi\omega\bmA(\Omega_{\theta}) + \grad_{y} V(\Omega_{\theta}))
    \cdot (\overline{\mi\omega\bmPsi(\Omega_{\theta}) + \grad_{y}\Phi(\Omega_{\theta})}) \d y.
\end{align*}
With the variable substitution $(\Id + \bmtheta)^{-1}: y\mapsto x$ and the identities \eqref{eq:DifChangeVariable}
related to $A_{\curl}$, $A_{\divv}$ and $V_{\grad}$, we rewrite the above form
on a fixed reference domain $\Omega = (\Id + \bmtheta)^{-1}\Omega_{\theta}$ as
\begin{align*}
  & \sfa(\Omega_{\theta})\big(\bmA,V;\bmPsi,\Phi\big)
  = \int_{\Omega} \dfrac{1}{\mu}\dfrac{(I+\grad\bmtheta)^{t}(I+\grad\bmtheta)}{|\det(I+\grad\bmtheta)|}
      \curl\bmA_{\curl} \cdot \curl \overline{\bmPsi_{\curl}} \d x \notag\\
  &+ \dfrac{1}{\mi\omega}\int_{\Omega}\sigma |\det(I+\grad\bmtheta)|(I+\grad\bmtheta)^{-1}(I+\grad\bmtheta)^{-t}
     \big(\mi\omega\bmA_{\curl} + \grad V_{\grad}\big)
     \cdot \big(\overline{\mi\omega\bmPsi_{\curl} + \grad\Phi_{\grad}}\big) \d x.
\end{align*}
If $(\bmB(\bmtheta), U(\bmtheta))$, $(\bmeta(\bmtheta), \chi(\bmtheta))$
are respectively the material derivatives of $(\bmA,V)$ and $(\bmPsi, \Phi)$,
then one can develop the above form with respect to $\bmtheta$
by considering the developments \cite{de2006conception}
\begin{subequations}\label{eq:thetaDevelop}
  \begin{align}
  & |\det(I+\grad\bmtheta)| = 1+\divv\bmtheta + o(\bmtheta), \\
  & (I + \grad\bmtheta)^{-1} = I - \grad\bmtheta + o(\bmtheta).
  \end{align}
\end{subequations}
Since $(\bmA_{\curl}(0),V_{\grad}(0)) = (\bmA(\Omega),V(\Omega))$, $(\bmPsi_{\curl}(0),\Phi_{\grad}(0) = \Phi(\Omega))$,
the terms of order zero with respect to $\bmtheta$ in the development give exactly $\sfa(\Omega)(\bmA,V;\bmPsi,\Phi)$,
while the first order terms with respect to $\bmtheta$ yield
\begin{align}\label{eq:middleI1I2}
  & \sfaprime(\Omega)(\bmA,V;\bmPsi,\Phi)
    = \sfa(\Omega)(\bmB(\bmtheta),U(\bmtheta);\bmPsi,\Phi)
    + \sfa(\Omega)(\bmA,V;\bmeta(\bmtheta),\chi(\bmtheta)) + \mathcal{I}_{1} + \mathcal{I}_{2}, \notag \\
  \text{with}\quad
  & \mathcal{I}_{1} = \int_{\Omega}\dfrac{1}{\mu}(-\divv \bmtheta + \grad\bmtheta + (\grad\bmtheta)^{t})
    \curl\bmA\cdot\curl\overline{\bmPsi} \d x, \notag \\
  & \mathcal{I}_{2} = \dfrac{1}{\mi\omega} \int_{\Omega}\sigma(\divv\bmtheta I - \grad\bmtheta - (\grad\bmtheta)^{t})
  \big(\mi\omega\bmA + \grad V\big)\cdot \big(\overline{\mi\omega\bmPsi + \grad\Phi}\big) \d x.
\end{align}
We will rewrite the volume integrals $\mathcal{I}_{1}$, $\mathcal{I}_{2}$ in terms of boundary integrals.
Using the differential identities \eqref{eq:VectorFormula}
and the fact that $(\bmA,V)$ satisfy the conditions \eqref{eq:lemmCompute}, one verifies
\begin{align*}
  & (-\divv\bmtheta I + \grad\bmtheta + (\grad\bmtheta)^{t})\curl \bmA
   = - \curl((\bmtheta\cdot\grad)\bmA + (\grad\theta)^{t}\bmA)
    + \grad(\bmtheta\cdot\curl\bmA) + \mu\sigma(\mi\omega\bmA + \grad V)\times \bmtheta.
\end{align*}
Hence
\begin{align*}
  & \mathcal{I}_{1} =
    - \int_{\Omega}\dfrac{1}{\mu}\curl((\bmtheta\cdot\grad)\bmA + (\grad\theta)^{t}\bmA) \cdot \curl\overline{\bmPsi} \d x
    + \mathcal{I}_{11} + \mathcal{I}_{22}, \notag \\
  & \text{where} \quad \mathcal{I}_{11} = \int_{\Omega}\dfrac{1}{\mu}\grad(\bmtheta\cdot\curl\bmA)\cdot \curl\overline{\bmPsi} \d x
    \quad \text{and} \quad
    \mathcal{I}_{12} = \int_{\Omega}\sigma \big((\mi\omega\bmA + \grad V)\times \bmtheta\big)\cdot \curl\overline{\bmPsi} \d x.
\end{align*}
By Stoke's theorem, one has
\begin{align*}
  \mathcal{I}_{11} = \int_{\Omega}\dfrac{1}{\mu} \divv\big((\bmtheta\cdot\curl\bmA)\curl\overline{\bmPsi}\big) \d x
    = \int_{\partial\Omega} \dfrac{1}{\mu}(\bmtheta\cdot\curl\bmA)(\bmn\cdot\curl\overline{\bmPsi}) \d s.
\end{align*}
By integration by parts (with use of differential identities \ref{eq:VectorFormula}), we verify
\begin{align*}
  \mathcal{I}_{12}
  = & - \dfrac{1}{\mi\omega}\int_{\Omega} \sigma \big\{
    (\divv\bmtheta I - \grad\bmtheta)(\mi\omega\bmA + \grad V)
    + (\bmtheta\cdot\grad)(\mi\omega\bmA + \grad V) \big\}\cdot (\overline{\mi\omega\bmPsi}) \d x \notag \\
  & + \dfrac{1}{\mi\omega}\int_{\partial\Omega} \sigma (\bmtheta\cdot \bmn)
    (\mi\omega\bmA + \grad V)\cdot(\overline{\mi\omega\bmPsi}) \d s.
\end{align*}
Therefore
\begin{align}\label{eq:finalI_1}
  \mathcal{I}_{1}
  = & - \int_{\Omega}\dfrac{1}{\mu}\curl((\bmtheta\cdot\grad)\bmA + (\grad\theta)^{t}\bmA) \cdot \curl\overline{\bmPsi} \d x \notag \\
  & - \dfrac{1}{\mi\omega}\int_{\Omega} \sigma \big\{
    (\divv\bmtheta I - \grad\bmtheta)(\mi\omega\bmA + \grad V)
    + (\bmtheta\cdot\grad)(\mi\omega\bmA + \grad V) \big\}\cdot (\overline{\mi\omega\bmPsi}) \d x \notag \\
  & + \int_{\partial\Omega} \dfrac{1}{\mu}(\bmtheta\cdot\curl\bmA)(\bmn\cdot\curl\overline{\bmPsi}) \d s
    + \dfrac{1}{\mi\omega}\int_{\partial\Omega} \sigma (\bmtheta\cdot \bmn)
    (\mi\omega\bmA + \grad V)\cdot(\overline{\mi\omega\bmPsi}) \d s.
\end{align}
Now we compute the term $\mathcal{I}_{2}$
\begin{align*}
  & \mathcal{I}_{2} =  \dfrac{1}{\mi\omega}\int_{\Omega} \sigma(\divv\bmtheta I - \grad\bmtheta - (\grad\bmtheta)^{t})
    (\mi\omega\bmA + \grad V)\cdot(\overline{\mi\omega\bmPsi}) \d x
    + \mathcal{I}_{21} + \mathcal{I}_{22}, \notag \\
  \text{with} \quad & \mathcal{I}_{21} = \dfrac{1}{\mi\omega}\int_{\Omega}\sigma \divv\bmtheta
    (\mi\omega\bmA + \grad V)\cdot \grad\overline{\Phi} \d x,
    \quad
    \mathcal{I}_{22} = \dfrac{1}{\mi\omega}\int_{\Omega} \sigma(- \grad\bmtheta - (\grad\bmtheta)^{t})
    (\mi\omega\bmA + \grad V)\cdot \grad\overline{\Phi} \d x.
\end{align*}
By integration by parts, one obtains
\begin{align*}
  \mathcal{I}_{21}
  = & \dfrac{1}{\mi\omega} \int_{\partial\Omega} \sigma(\bmtheta\cdot\bmn)
    (\mi\omega\bmA + \grad V)\cdot\grad\overline{\Phi} \d s
    - \dfrac{1}{\mi\omega}\int_{\Omega} \mi\omega\sigma (\bmtheta\cdot\grad)\bmA\cdot \grad\overline{\Phi} \d x \notag \\
  & - \dfrac{1}{\mi\omega}\int_{\Omega} \sigma D^{2}\overline{\Phi}(\mi\omega\bmA + \grad V)\cdot\bmtheta \d x
    - \dfrac{1}{\mi\omega}\int_{\Omega} \sigma D^{2}V\bmtheta \cdot\grad\overline{\Phi} \d x.
\end{align*}
Using integration by parts and the fact that $\divv(\sigma(\mi\omega\bmA + \grad V))=0$
obtained by applying the divergence operator to \eqref{eq:lemmCompute}$_1$, one verifies
\begin{align*}
  \mathcal{I}_{22}
  = & - \dfrac{1}{\mi\omega}\int_{\Omega}
    \sigma(\grad\bmtheta)^{t} (\mi\omega\bmA + \grad V)\cdot \grad\overline{\Phi} \d x
    + \dfrac{1}{\mi\omega}\int_{\Omega}D^{2}\overline{\Phi}(\mi\omega\bmA + \grad V)\cdot \bmtheta \d x.
\end{align*}
From the differential identities \eqref{eq:VectorFormula}, one deduces also that
\begin{align*}
  D^{2}V\bmtheta + (\grad\bmtheta)^{t}\grad V = (\bmtheta\cdot\grad)\grad V + (\grad\bmtheta)^{t}\grad V
  = \grad(\bmtheta\cdot\grad V).
\end{align*}
The above equalities yield
\begin{align}\label{eq:finalI_2}
  \mathcal{I}_{2}  = & \mathcal{I}_{21} + \mathcal{I}_{22} \notag \\
  = & \dfrac{1}{\mi\omega}\int_{\Omega} \sigma(\divv\bmtheta I - \grad\bmtheta - (\grad\bmtheta)^{t})
    (\mi\omega\bmA + \grad V)\cdot(\overline{\mi\omega\bmPsi}) \d x
    + \dfrac{1}{\mi\omega} \int_{\partial\Omega} \sigma(\bmtheta\cdot\bmn)
    (\mi\omega\bmA + \grad V)\cdot\grad\overline{\Phi} \d s \notag \\
  & - \dfrac{1}{\mi\omega}\int_{\Omega} \sigma
    \mi\omega\big((\bmtheta\cdot\grad)\bmA + (\grad\bmtheta)^{t}\bmA\big)\cdot \grad\overline{\Phi} \d x
    - \dfrac{1}{\mi\omega}\int_{\Omega} \sigma \grad(\bmtheta\cdot\grad V) \cdot\grad\overline{\Phi} \d x.
\end{align}
\eqref{eq:finalI_1}, \eqref{eq:finalI_2} and the fact that $\sigma(\mi\omega \bmA + \grad V)\cdot\bmn = 0$ on $\partial\Omega$ imply
\begin{align} \label{eq:finalI}
  \mathcal{I}_{1} + \mathcal{I}_{2}
   = & \sfa(\Omega)\Big(-(\bmtheta\cdot\grad)\bmA - (\grad\theta)^{t}\bmA,- (\bmtheta\cdot\grad V);\bmPsi,\Phi\Big) \notag \\
   & + \int_{\partial\Omega} \dfrac{1}{\mu}(\bmtheta\cdot\curl\bmA)(\bmn\cdot\curl\overline{\bmPsi}) \d s
    + \dfrac{1}{\mi\omega}\int_{\partial\Omega} \sigma (\bmtheta\cdot \bmn)
    (\mi\omega\bmA_{\tau} + \grad_{\tau} v)\cdot(\overline{\mi\omega\bmPsi_{\tau} + \grad\Phi_{\tau}}) \d s.
\end{align}
From \eqref{eq:middleI1I2}, \eqref{eq:finalI} and the definition of shape derivatives
\eqref{eq:dAdVmaterial}, one concludes the result \eqref{eq:alphaDerivative}.
\end{proof}

\section{Proof of Proposition~\ref{prop:impedanceDerivative3d}}
We give the proof of the stated theorem \ref{prop:impedanceDerivative3d}
\begin{proof}
Taking $(\bmPsi, \Phi) = (\bmB_{k}(\bmtheta),U_{k}(\bmtheta)) \in \X$
in the adjoint problem \eqref{eq:fvAdjoint3d} yields
\begin{align*}
  & \mathcal{S}^{*}(\bmP_{l},W_{l};\bmB_{k}(\bmtheta),U_{k}(\bmtheta)) = L^{*}(\bmB_{k}(\bmtheta),U_{k}(\bmtheta)).
\end{align*}
On the other hand, taking $(\bmvarPsi,\varPhi) = (\bmP_{l}, W_{l})$
in the variational formulation \eqref{eq:fvDerivative3dBis}
for the material derivatives $(\bmB_{k}(\bmtheta), U_{k}(\bmtheta))$ implies
\begin{align*}
  \mathcal{S}(\bmB_{k}(\bmtheta),U_{k}(\bmtheta);\bmP_{l}, W_{l}) = \mathcal{L}(\bmP_{l}, W_{l}).
\end{align*}
Since
\begin{align*}
  \overline{\mathcal{S}^{*}(\bmP_{l},W_{l};\bmB_{k}(\bmtheta),U_{k}(\bmtheta))} = \mathcal{S}(\bmB_{k}(\bmtheta),U_{k}(\bmtheta);\bmP_{l}, W_{l})
\end{align*}
with the fact that $\divv \bmP_{l}=0$, one obtains
\begin{align*}
  & \overline{L^{*}(\bmB_{k}(\bmtheta),U_{k}(\bmtheta))} = \mathcal{L}(\bmP_{l}, W_{l}).
 \notag \\
  & = \int_{\Omega_{d}}\dfrac{1}{\mu}
    \curl\big((\bmtheta\cdot\grad)\bmA_{k} + (\grad\bmtheta)^{t}\bmA_{k}\big)\cdot\curl\overline{\bmP_{l}}\d x \notag \\
  & + \dfrac{1}{\mi\omega}\int_{\Omega_{\mathcal{C}}}\sigma\bigg(\mi\omega\big((\bmtheta\cdot\grad)\bmA_{k}+(\grad\bmtheta)^{t}\bmA_{k}\big)
    + \grad(\bmtheta\cdot\grad V_{k})\bigg)\cdot(\overline{\mi\omega \bmP_{l}+\grad W_{l}}) \d x\notag \\
  & + \int_{\Gamma}\left[\dfrac{1}{\mu}\right](\bmtheta\cdot\bmn)(\bmn\cdot \curl\bmA_{k}) (\bmn\cdot\curl\overline{\bmP_{l}})\d s
    + \dfrac{1}{\mi\omega} \int_{\Gamma}(\bmtheta\cdot\bmn)[\sigma](\mi\omega \bmA_{k\tau}
    + \grad_{\tau}V_{k})\cdot(\overline{\mi\omega \bmP_{l\tau}+\grad_{\tau}W_{l}})\d s.
\end{align*}
In $\Omega\backslash\Gamma$ one verifies
\begin{align*}
  & (\bmtheta\cdot\grad)\bmA_{k} + (\grad\bmtheta)^{t}\bmA_{k}
    = \curl\bmA_{k}\times\bmtheta + \grad(\bmtheta\cdot\bmA_{k}), \notag \\
  & \curl\big((\bmtheta\cdot\grad)\bmA_{k} + (\grad\bmtheta)^{t}\bmA_{k}\big)
    = \curl\big(\curl\bmA_{k}\times\bmtheta\big).
\end{align*}
Thus, considering \eqref{eq:PotentialAdj}$_4$ and \eqref{eq:PotentialAdj}$_5$, we compute
\begin{align}\label{eq:suan0}
  & \overline{L^{*}(\bmB_{k}(\bmtheta),U_{k}(\bmtheta))} \notag \\
  &  = \mathcal{I}
  + \int_{\Gamma}\left[\dfrac{1}{\mu}\right](\bmtheta\cdot\bmn)(\bmn\cdot \curl\bmA_{k}) (\bmn\cdot\curl\overline{\bmP_{l}})\d s
  + \dfrac{1}{\mi\omega} \int_{\Gamma}(\bmtheta\cdot\bmn)[\sigma](\mi\omega \bmA_{k\tau}
    + \grad_{\tau}V_{k})\cdot(\overline{\mi\omega \bmP_{l\tau}+\grad_{\tau}W_{l}})\d s, \notag \\
  & \text{with} \quad
  \mathcal{I} = \int_{\Omega_{d}}\dfrac{1}{\mu}
    \curl\big(\curl\bmA_{k}\times\bmtheta\big)\cdot\curl\overline{\bmP_{l}}\d x
    + \dfrac{1}{\mi\omega}\int_{\Omega_{\mathcal{C}}}\sigma\mi\omega\big(\curl\bmA_{k}\times\bmtheta\big)
    \cdot(\overline{\mi\omega \bmP_{l}+\grad W_{l}}) \d x.
\end{align}
We remind that $(\curl\bmA_{k} \times \bmtheta)$ belongs to $X(\Omega)$.
We multiply \eqref{eq:PotentialAdj}$_1$ by
$\big(\overline{\curl A_{k} \times \bmtheta}\big)$,
integrate by parts and then take the complex conjugate, which implies
\begin{align}\label{eq:suanI}
  \mathcal{I} & =
  \int_{\Omega_{d}}(\dfrac{1}{\mu}-\dfrac{1}{\mu^{0}})\curl \bmA_{l}^{0}
    \cdot \curl(\curl \bmA_{k} \times \bmtheta) \d x
    - \dfrac{1}{\mi\omega}\int_{\Omega_{d}}[\sigma](\mi\omega \bmA_{l}^{0} + \grad V_{l}^{0})
    \cdot\big(\mi\omega (\curl A_{k} \times \bmtheta)\big)\d x \notag \\
  & \quad + \int_{\Gamma} \left[\dfrac{1}{\mu}\curl\overline{\bmP_{l}}
    \cdot \big((\curl \bmA_{k} \times \bmtheta)\times \bmn\big)\right]\d s
  + \int_{\Gamma}\left[\dfrac{1}{\mu}\right] \curl\overline{\bmA_{l}^{0}}
    \cdot \big((\curl \bmA_{k} \times \bmtheta)\times \bmn\big) \d s \notag \\
  & = \int_{\Omega_{d}}\left[\dfrac{1}{\mu}\right]\curl \bmA_{l}^{0}
    \cdot \curl\big((\bmtheta\cdot\grad)\bmA_{k} + (\grad\bmtheta)^{t}\bmA_{k}\big) \d x \notag \\
  & \quad - \dfrac{1}{\mi\omega}\int_{\Omega_{d}}[\sigma](\mi\omega \bmA_{l}^{0} + \grad V_{l}^{0})
    \cdot\bigg(\mi\omega \big((\bmtheta\cdot\grad)\bmA_{k} + (\grad\bmtheta)^{t}\bmA_{k}\big)
    + \grad(\bmtheta\cdot \grad V_{k}) \bigg)\d x \notag \\
  & \quad - \int_{\Gamma} (\bmtheta\cdot \bmn) [\mu] \left(\dfrac{1}{\mu}\curl \bmA_{k} \times \bmn \right)
    \cdot \left(\dfrac{1}{\mu^{0}} (\curl\overline{\bmP_{l}})_{+}\times \bmn\right) \d s
\end{align}
The last equality is due to the transmission conditions
\eqref{eq:PotentialAdj}$_2$ -- \eqref{eq:PotentialAdj}$_3$ for $\bmP_{l}$
and those for $\bmA_{k}$ on $\Gamma$:
$[\bmn\cdot\curl \bmA] = [\mu^{-1} \bmn\times \curl \bmA] = 0$.
\eqref{eq:suan0} and \eqref{eq:suanI} imply
\begin{align}\label{eq:AdjTestDifferenceShape}
  & \overline{L^{*}(\bmB_{k}(\bmtheta),U_{k}(\bmtheta))}
    - \int_{\Omega_{d}}\left[\dfrac{1}{\mu}\right]\curl \bmA_{l}^{0}
    \cdot \curl\big((\bmtheta\cdot\grad)\bmA_{k} + (\grad\bmtheta)^{t}\bmA_{k}\big) \d x \notag \\
  & \quad + \dfrac{1}{\mi\omega}\int_{\Omega_{d}}[\sigma](\mi\omega \bmA_{l}^{0} + \grad V_{l}^{0})
    \cdot\bigg(\mi\omega \big((\bmtheta\cdot\grad)\bmA_{k} + (\grad\bmtheta)^{t}\bmA_{k}\big)
    + \grad(\bmtheta\cdot \grad V_{k}) \bigg)\d x \notag \\
  & = \int_{\Gamma}(\bmtheta\cdot \bmn)\bigg\{
    \left[\dfrac{1}{\mu}\right](\bmn\cdot \curl\bmA_{k}) (\bmn\cdot\curl\overline{\bmP_{l}})
    - [\mu] \left(\dfrac{1}{\mu}\curl \bmA_{k} \times \bmn \right)
    \cdot \left(\dfrac{1}{\mu^{0}} (\curl\overline{\bmP_{l}})_{+}\times \bmn\right) \notag \\
  & \qquad \qquad \qquad + \dfrac{1}{\mi\omega} [\sigma](\mi\omega \bmA_{k\tau}
    + \grad_{\tau}V_{k})\cdot(\overline{\mi\omega \bmP_{l\tau}+\grad_{\tau}W_{l}}) \bigg\}\d s.
\end{align}

Considering the definition of $L^{*}(\cdot,\cdot)$, we substitute the above integral
\eqref{eq:AdjTestDifferenceShape} in the expression of shape derivative of
$\triangle Z_{kl}$ \eqref{eq:impedanceDerivative3dBis} and finally obtain \eqref{eq:impedanceDerivative3dFinal}.
\end{proof}

\bibliographystyle{plain} \bibliography{biblio}

\end{document}